\documentclass[12pt,a4paper,fleqn]{article}
\usepackage{a4wide,amsfonts,amsmath,latexsym,amssymb,euscript,graphicx,units,mathrsfs}

\usepackage{graphicx}
\usepackage{color}
\usepackage{amssymb}
\usepackage{amssymb}
\usepackage[T1]{fontenc}
\usepackage{latexsym}
\usepackage{xypic}
\usepackage{eufrak}
\usepackage{euscript}
\usepackage{amsfonts,amsmath,amsthm}
\usepackage{verbatim}
\usepackage{fancyhdr}
\usepackage[english]{babel}
\usepackage{mathrsfs}
\usepackage{units}

\newtheorem{prop}{Proposition}[section]

\newtheorem{lemma}[prop]{Lemma}
\newtheorem{thm}[prop]{Theorem}
\newtheorem{defi}[prop]{Definition}

\theoremstyle{definition}
\newtheorem{rem}[prop]{Remark}

\newtheorem{example}[prop]{Example}
\newtheorem{examples}[prop]{Examples}
\newtheorem{as}[prop]{Assumption}
\newtheorem{prob}[prop]{Problem}

\renewcommand{\geq}{\geqslant}
\def\leq{\leqslant}

\newcommand{\R}{\mathbb{R}}

\newcommand{\WI}{Wiener-It\^o }

\newcommand{\Us}{$U$-statistic}
\newcommand{\Uss}{$U$-statistics}

\newcommand{\Cpt}{\mathcal{K}}

\def\var{{\mathbb{Var}}}

\def\1{{\mathbf{1}}}

\def\1{{\mathbf{1}}}
\def\0.5{{\frac{1}{2}}}
\def\var{{\rm{Var}}}

\usepackage{fullpage}

\usepackage{hyperref}

\newcommand{\leb}{\ell}

\renewcommand{\phi}{\varphi}
\renewcommand{\kappa}{\varkappa}

\newcommand{\mrk}{\underline}

\newcommand{\Zi}{Z_{\infty}}

\newcommand{\fact }{\overline}

\newcommand{\norm}{\mathscr{N}(0,1)}

\newlength{\querylen}
\usepackage{fancybox}

\newcommand{\bfx}{\mathbf{x}}
\newcommand{\bfy}{\mathbf{y}}
\newcommand{\bfz}{\mathbf{z}}

\newcommand{\bft}{\mathbf{t}}

\newcommand{\bfs}{\mathbf{s}}
\newcommand{\bfm}{\mathbf{m}}

\begin{document}

\begin{center}
{\Large{\bf Fine Gaussian fluctuations on the Poisson space II: rescaled kernels, marked processes and geometric  $U$-statistics}}
\normalsize
\\~\\ by Rapha\"el Lachi\`eze-Rey and Giovanni Peccati \\ {\it Universit\'e Paris Descartes} and {\it  Universit\'e du Luxembourg}\\~\\
\end{center}

{\small \noindent {\bf Abstract}: Continuing the analysis initiated in Lachi\`eze-Rey and Peccati (2011), we use contraction operators to study the normal approximation of random variables having the form of a $U$-statistic written on the points in the support of a random Poisson measure. Applications are provided: to subgraph counting, to boolean models and to coverage of random networks.

\noindent {\bf Key words}: Central Limit Theorems; Contractions; Malliavin Calculus; Poisson Space; Stein's Method; Stochastic Geometry; $U$-statistics; Wasserstein Distance; Wiener Chaos \\

\noindent {\bf 2000 Mathematics Subject Classification: 60H07, 60F05, 60G55, 60D05.} 

\tableofcontents

\section{Introduction}

This paper is a direct continuation of \cite{LacPec}, where we have investigated the normal approximation of random variables belonging to a fixed sum of Poisson Wiener chaoses, with special emphasis on $U$-statistics living on the support of a Poisson random measure. As we will see below, strong motivations come from a fundamental paper by Reitzner and Schulte \cite{lesmathias}, where a first connection between Malliavin operators and limit theorems in stochastic geometry was established, and several bounds were obtained via an extensive use of product formulae for multiple integrals (see e.g. \cite[Chapter 6]{PeTa}) -- with applications e.g. to statistics based on random graphs, or to the Gaussian fluctuations of the intrinsic volumes of $k$-flats intersections of a convex body. 

\medskip

As discussed in more detail in Section \ref{s:framework}, the main contribution of \cite{LacPec} was the derivation of general upper bounds, expressed in terms of {\it contraction operators}, on the Wasserstein distance between the law of a finitely chaotic random variable and of a centered Gaussian distribution. As further demonstrated by the examples developed in this paper, we believe that contractions are the most natural object for dealing with the normal approximation of random variables having a finite chaotic decomposition (like $U$-statistics): for instance, our bounds yield conditions for central limit theorems that are in many instances necessary and sufficient, and automatically imply joint CLTs (with bounds) for chaotic components of different orders. 

\begin{rem}{\rm As a complement to the analysis developed in \cite{LacPec}, one should note that bounds based on contraction operators (such as the one appearing on formula (\ref{e:B3}) below) are considerably simpler than those obtained by developing expectations of Malliavin operators via diagram-type multiplication formulae (such as the one stated e.g. in \cite[Theorem 6.1.1]{PeTa}). One combinatorial reason for this phenomenon is that, in the jargon of diagram formulae (see \cite[Chapter 6]{PeTa}), computing norms of contractions requires one to assess integrals labeled by partitions having blocks of size either two or four that can be arranged in a circular way (see \cite[p. 49]{PeTa} for definitions and illustrations), whereas bounds based on diagram computations contain integrals associated with general noncircular partitions having possibly blocks of size 3.
}
\end{rem}

\medskip

The theoretical findings of \cite{LacPec} were applied to characterize the Gaussian fluctuations of edge counting statistics associated with general random graphs, as well as to describe a `Gaussian-to-Poisson' transition for random graphs with sparse connections. When applied to the so-called {\it disk graphs} (see e.g. Penrose \cite{penrosebook}), the results about Poisson limits yield Poissonized versions of classic findings, e.g. by Jammalamadaka and Janson \cite{JJ} and Silverman and Brown \cite{SiBr}. Note that Poisson approximations were established in \cite{LacPec} by the method of `diagram formulae' (see \cite[Chapter 7]{PeTa}), and then further refined (in a much more general framework) in \cite{pec2011} by combining the Malliavin calculus of variations and a classic version of the Chen-Stein method.

\medskip

The principal aim of the present paper is to extend the results proved in \cite{LacPec} in order to study the Gaussian fluctuations of $U$-statistics (of a general order) with rescaled kernels, based on the points of a {\it marked point process}. As shown in Section \ref{s:applications}, marked point processes emerge naturally in a number of applications, as for instance those involving Boolean models. Another contribution of the present work is an exhaustive characterization of the fluctuations of geometric $U$-statistics, thus completing the analysis initiated in \cite[Section 5]{lesmathias}: this point will be dealt with in Section \ref{sec:geom-Ustat} by applying the classic theory of Hoeffding decompositions for symmetric $U$-statistics based on i.i.d. samples (see e.g. \cite{Vitale}). 

\medskip

\begin{rem}{\rm Some of the central limit theorems deduced in the present paper as well as in \cite{LacPec} could alternatively be obtained by combining the results of \cite{BhGh, JJ} with a standard poissonization argument. Unlike our techniques, this approach would {\it not} yield explicit bounds on the speed of convergence in the Wasserstein distance.

}
\end{rem}

\medskip

We stress that the main results of \cite{LacPec} and \cite{lesmathias}, that constitute the theoretical backbone of our analysis, were obtained by means of the techniques developed in \cite{PSTU, PecZheng}, that  were in turn based on a combination of the Malliavin calculus of variations and of the so-called {\it Stein's method} for normal approximations (see e.g. \cite{ChenGoldShao} for a general reference on this topic). Other remarkable contributions to the line of research to which the present paper belongs are the following. In \cite{DFR}, Decreusefond {\it et al.} present applications of the findings of \cite{PSTU} to the Gaussian fluctuation of statistics based on random graphs on a torus; reference \cite{minh} contains several multidimensional extensions of the theory initiated in \cite{lesmathias}; the paper \cite{pec2011} contains some general bounds associated with the Poisson approximation of the integer-valued functionals of a Poisson measure (further applied in \cite{SchTh2012} to the study of order statistics); references \cite{Schulte} and \cite{SchTh} use some of the techniques introduced in \cite{PSTU, PecZheng} in order to deal, respectively, with Poisson-Voronoi approximations, and with the asymptotic fluctuations of Poisson $k$-flat processes; further applications to $U$-statistics are discussed in \cite{LasPenSchTha}.

\medskip

The remainder of the paper is organized as follows. Section \ref{s:framework} presents some background material as well as a description of the main problems addressed in the paper. In Section \ref{s:sub} we discuss a preliminary example about subgraph counting, which is meant to familiarize the reader with our approach. Section \ref{s:technical} contains fundamental estimates involving contraction operators; Section \ref{s:finite} and Section \ref{s:stat} deal, respectively, with random variables living in a finite sum of Wiener chaoses and with rescaled $U$-statistics with a stationary kernel; Section \ref{sec:geom-Ustat} contains the announced characterization of the asymptotic behavior of geometric $U$-statistics, whereas Section \ref{s:applications} is devoted to applications - in particular, to the Boolean model, and to random graphs.

\section{Framework and goals}\label{s:framework}

This section contains a general description of the mathematical framework of the paper, as well as of the main problems that are addressed in the sections to follow. We start with a synthetic description of the results proved in \cite{LacPec}, which constitute the theoretical backbone of our analysis. One should also note that reference \cite{LacPec} uses in a fundamental way the Malliavin calculus techniques developed in \cite{PSTU, PecZheng}.

\subsection{Asymptotic normality on the Poisson chaos: a quick overview }

\begin{itemize}

\item[\bf (i)] {\bf (Framework)} Throughout this section, we shall consider a measure space of the type $(Z,\mathscr{Z},\mu) $, where
$Z$ is a Borel space, $\mathscr{Z}$ is the associated Borel $\sigma$-field, and $\mu$ is a $\sigma$-finite
Borel measure with no atoms.  We shall denote by $\eta = \{\eta(B) : B\in \mathscr{Z}, \, \mu(B)<\infty\}$ a {\it Poisson measure} on $(Z,\mathscr{Z})$ with control measure $\mu$, which we assume to be defined on some adequate probability space $(\Omega,\mathscr{F},P)$. We shall assume that $\mathscr{F}$ is the $P$-completion of the $\sigma$-field generated by $\eta$, so that $L^2(P) = L^2(\Omega,\mathscr{F},P)$ coincides with the space of square-integral functionals of $\eta$. We  write $\hat{\eta} = \eta - \mu$ for the associated compensated Poisson measure. For every $k\geq 1$, the symbol $L^2(Z^k; \mu^k)$ (that we shall sometimes shorten to $L^2(Z^k)$ or $L^2(\mu^k)$ if there is no ambiguity) stands for the space of measurable functions on $Z^k$ that are square-integrable with respect to $\mu^k$. As usual, $L_s^2(Z^k; \mu^k) = L_{s}^2(Z^k) = L_s^2( \mu^k)$ is the subspace of $L^2(\mu^k)$ composed of functions that are $\mu^k$-almost everywhere symmetric. We shall also adopt the special notation $L^{p,p'}(Z^k)=L^p(Z^k)\cap L^{p'}(Z^k)$, for $p,p' \geq 1$.

\item[\bf (ii)] {\bf (Multiple integrals and chaos)}  For  $f\in L_s^2(\mu^q), q\geq 1$, we denote by $I_q(f)$
the {\it multiple Wiener-It\^o integral}, of order $q$, of $f$ with respect to $\hat{\eta}$, that is:
\begin{equation}
\label{eq:def-WI}
I_{q}(f)=\int_{(Z^q)'}f(\bfx)d\hat\eta^{\otimes q}(\bfx).
\end{equation}
Where the symbol $(Z^q)'$ indicates that all diagonal sets have been eliminated from the domain of integration. The reader is referred for instance to \cite[Chapter 5]{PeTa} for a complete discussion of multiple Wiener-It\^o integrals and their properties. These random variables play a fundamental role in our analysis, since every square-integrable functional $F=F(\eta)$ can be decomposed in a (possibly infinite) sum of \WI integrals. This feature, known as ``chaotic representation property'', is the object of the next statement.

\begin{prop}\label{P: MWIchaos} Every random variable $F\in L^2(P)$
admits a (unique) chaotic decomposition of the type
\begin{equation} \label{chaos}
F= E[F] + \sum_{i = 1}^{\infty} I_i(f_i),
\end{equation}
where the series converges in $L^2(P)$ and, for each $i\geq 1$, the kernel $f_i$ is an element
of $L^2_s(Z^i,\mu^i)$.
\end{prop}

\item[\bf (iii)] {\bf (Wasserstein distance)} Let $N$ be a centered Gaussian variable with variance $\sigma^2>0$. In this paper, we will be interested in assessing the distance between the law of $N$ and that of a random variable $F\in L^2(P)$ having the form 
\begin{equation}\label{e:genrv}
F = E[F]  + \sum_{i=1}^k I_{q_i}(f_i),
\end{equation} 
where $k \geq 1$, $1\leq q_{1}<q_{2}<\dots < q_{k}$ are integers and, for every $1 \leq i \leq k$, $f_{i}$ is a non-zero element of $L_{s}^2(\mu^{q_{i}})$. The Wasserstein distance between the law of $F$ and the law of $N$ is defined as 
\begin{equation}
\label{eq:wass}
d_{W}(F,N)=\sup_{h\in{\rm Lip}_{1}}E|h(F)-h(N)|,
\end{equation}
where ${\rm Lip}_{1}$ stands for the class of Lipschitz functions with Lipschitz constant $\leq 1$.  In the forthcoming Theorem \ref{t:main2}, we shall collect results from \cite{LacPec}, allowing one to assess $d_W(F,N)$ by means of expressions involving the kernels of the \WI expansion of $F$.

\item[\bf (iv)] {\bf (Contractions)} The main bounds evaluated in this paper are expressed in terms of the {\it contractions} of the kernels $f_{i}$ (see e.g. \cite{LacPec, PSTU} for full details).
Given two functions $h \in L_s^2(Z^p ; \mu^p)$ and $g\in L_s^2(Z^q; \mu^q)$ (for some $p,q \geq 1$), we shall use the following notation:  for every $0 \leq l \leq r \leq \min(q,p)$, and whenever it is well-defined (see the discussion at Point (v) below), the {\it }contraction of $g$ and $h$ is the function in $p+q-r-l$ variables given by
  
\begin{eqnarray}
& & h \star_r^l
g(x_{1},\dots,x_{p-r},x'_{1},\dots,x'_{q-r},y_{1},\dots,y_{r-l}) \label{contraction} \\
&=& \int_{Z^l} \mu^l(dz_1,...,dz_l)
h(x_{1},\dots,x_{p-r},y_{1},\dots,y_{r-l},z_{1},\dots,z_{l}) \nonumber \\
& & \quad\quad\quad\quad\quad\quad\quad\quad\quad\quad\quad\quad \times g(x'_{1},\dots,x'_{q-r},y_{1},\dots,y_{r-l},z_{1},\dots,z_{l}). \nonumber
\end{eqnarray}

In particular, if $p=q$ one has that $h\star_p^p g =\langle h,g \rangle_{L^2(Z^p)}$, the usual scalar product of $g$ and $h$. For the sake of brevity, we shall often use multi-dimensional variables, represented by a bold letter and indexed by their dimension: in this way, (\ref{contraction}) can be rewritten as 
\begin{equation*}
h\star^{l}_{r}g(\bfx_{p-r},\bfx'_{q-r}\bfy_{r-l})=\int_{Z^l}h(\bfx_{p-r},\bfy_{r-l},\bfz_{l})g(\bfx'_{q-r},\bfy_{r-l},\bfz_{l})d\mu^l,
\end{equation*}
for $\bfx_{q-r}\in Z^{q-r},\, \bfx'_{p-r}\in Z^{p-r},\,\bfy_{r-l}\in Z^{r-l}$. Finally, we observe that, if $h\star_{r}^l g$ is square-integrable, then its squared $L^2$ norm is given by the following iterated integral:
\begin{eqnarray}
\label{eq:norm-contraction}
&&\|h\star_{r}^l g\|^2_{L^2(Z^{p+q-r-l};\mu^{p+q-r-l})}\\ &&=\int_{Z^{p+q-r+l}}\!\!\!h(\bfx_{p-r},\bfy_{r-l},\bfz_{l})h(\bfx_{p-r},\bfy_{r-l},\bfz'_{l})
\nonumber g(\bfx'_{q-r},\bfy_{r-l},\bfz_{l})g(\bfx'_{q-r},\bfy_{r-l},\bfz'_{l})d\mu^{p+q-r+l}.
\end{eqnarray}

\item[\bf (v)] {\bf (Assumptions on kernels)} The following assumption will be always satisfied by the random variables considered in this paper.
\begin{as}\label{a:tech}{\rm Every random variable of the type (\ref{e:genrv}) considered in the sequel of this paper is such that the following properties (1)-(3) are verified.
\begin{enumerate}

\item[(1)] For every $i=1,...,d$ and every $r=1,..., q_i$, the kernel $f_i\star_{q_i}^{q_i-r} f_i $ is an element of $L^2(\mu^{r})$.

\item[(2)] For every $i$ such that $q_i\geq 2$, every contraction of the type $(z_1,...,z_{2q_i - r- l})\mapsto |f_i|\star_r^l |f_i| (z_1,...,z_{2q_i - r- l})$ is well-defined and finite for every $r=1,...,q_i$, every $l=1,...,r$ and every $(z_1,...,z_{2q_i - r- l})\in Z^{2q_i-r-l}$.

\item[(3)] For every $ i,j=1,...,d$ such that $\max(q_i,q_j) >1$, for every $k = |q_i - q_j| \vee 1,..., q_i+q_j-2$ and every $(r,l)$ verifying $k = q_i+q_j -2-r-l$, 
\[
\int_Z \left[\sqrt{ \int_{Z^k} (f_i(z,\cdot)\star_r^l f_j(z,\cdot))^2 \,\,d\mu^k  }\,\,\,\right]\mu(dz)<\infty,
\]
where, for every fixed $z\in Z$, the symbol $f_i(z,\cdot)$ denotes the mapping $(z_1,...,z_{q-1}) \mapsto f_i(z,z_1,...,z_{q-1})$.
\end{enumerate}

}
\end{as}

\begin{rem}{\rm According to \cite[Lemma 2.9 and Remark 2.10]{PecZheng}, Point (1) in Assumption \ref{a:tech} implies that the following properties (a)-(c) are verified:

\begin{enumerate}

\item[(a)] for every $1\leq i<j\leq k$, for every $r=1,...,q_i\wedge q_j$ and every $l=1,...,r$, the contraction $f_i \star_r^l f_j$ is a well-defined element of $L^2(\mu^{q_i+q_j-r-l})$;

\item[(b)] for every $1\leq i\leq j\leq k$ and every $r=1,...,q_i$, $f_i\star_r^0 f_j$ is an element of $L^2(\mu^{q_i+q_j-r})$;

\item[(c)] for every $i=1,...,k$, for every $r=1,...,q_i$, and every $l=1,...,r\wedge (q_i-1)$, the kernel $f_i\star_r^l f_i $ is a well-defined element of $L^2(\mu^{2q_i-r-l})$.
\end{enumerate}

In particular, every random variable $F$ verifying Assumption \ref{a:tech} is such that $I_{q_i}(f_i)^2 \in L^2(P)$ for every $i=1,...,k$, yielding in turn that $E[F^4]<\infty$. Note that Assumption \ref{a:tech} is verified whenever the kernels $f_i$ are bounded functions with support in a rectangle of the type $B\times \dots\times B$, $\mu(B) <\infty$. 
}
\end{rem}

\item[\bf (vi)] {\bf (The bound $B_3$)} For
  $F$ as in (\ref{e:genrv}), let $\sigma^2 = E[F^2]$. Following \cite{LacPec}, we set
\begin{eqnarray}\label{e:B3}
&& B_3(F;\sigma) = \frac{1}{\sigma}\left\{ \max_{(\ast)} \ \|f_i\star_r^l f_j\|_{L^2(\mu^{q_i+q_j-r-l})} + \max_{i=1,...,k} \|f_i\|^2_{L^4(\mu^{q_i})}\right\},
\end{eqnarray}
where $\displaystyle{\max_{(\ast)}}$ ranges over all quadruples $(i,j,r,l)$ such that  $1 \leq l \leq r \leq q_{i} \leq q_{j}$ ($i,j \leq k$) and $l\neq q_{j}$ (in particular, quadruples such that $l =r= q_{i}=q_{j}=1$ do not appear in the argument of $\displaystyle{\max_{(\ast)}}$).


\item[\bf (vii)] {\bf (Normal approximations)} Now fix integers $k\geq 1$ and $1\leq q_1<q_2<...<q_k$, and consider a family $\{F_{\lambda} : \lambda>0\}$ of random variables with the form
\begin{equation}\label{e:genseq}
F_{\lambda} =\sum_{i=1}^k I_{q_i}(f_{i,\lambda}), \quad \lambda>0,
\end{equation}
with kernels $f_{i,\lambda}$ verifying Assumption \ref{a:tech} for each $\lambda$. We also use the following additional notation: (i) $\sigma^2_{\lambda} = E[(F_{\lambda})^2]$, (ii) $F_{i,\lambda} =   I_{q_i}(f_{i,\lambda})$, $i=1,...,d$, and (iii) $\sigma_{i,\lambda}^2 = E[(F_{i,\lambda})^2]$. The next result collects some crucial findings from \cite{LacPec}.

\begin{thm}{\bf (See \cite{LacPec})}\label{t:main2} Let $\{F_{\lambda} \}$ be a collection of random variables as in {\rm (\ref{e:genseq})}, where the integer $k$ does not depend on $\lambda$. Assume that there exists $\sigma^2>0$ such that $\lim_{\lambda \to \infty}\sigma_{\lambda}^2 = \sigma^2$. Let $N\sim\mathscr{N}(0, \sigma^2)$.

\begin{enumerate}

\item[\rm 1.] For every $\lambda$, one has the estimate
\begin{equation}\label{e:bnd}
d_W(F_{\lambda},N)\leq C_0\times B_3(F_{\lambda};\sigma_\lambda) +\frac{\sqrt{2/\pi}}{\sigma_\lambda\vee \sigma} |\sigma_{\lambda}^2 - \sigma^2|,
\end{equation} 
where $C_0$ is some constant depending uniquely on $k$. In particular, if $B_3(F_{\lambda};\sigma_\lambda)\to 0$, as $\lambda\to \infty$, then $d_W(F_{\lambda},N)\to 0$ and therefore $F_{\lambda} \stackrel{\rm Law}{\to} N$.

\item[\rm 2.] Assume that $f_{i,\lambda}\geq 0$ for every $i,\lambda$, and also that the family $\{ F_{\lambda}^4 :\lambda>0\}$ is uniformly integrable. Then, the following conditions {\rm (a)--(c)} are equivalent, as $n\to \infty$: {\rm (a)} $d_W(F_{\lambda},N)\to 0$, {\rm (b)}  $B_3(F_{\lambda};\sigma_\lambda)\to 0$, and {\rm (c) }$E[F_{\lambda}^4] - 3\sigma_{\lambda}^4\to 0$.

\item[\rm 3.]  Let $\phi : \R^k \to \R$ be a  thrice differentiable function with bounded second and first derivatives. Then, if  $B_3(F_{\lambda};\sigma_\lambda)\to 0$  \[
E[ \phi(F_{1,\lambda},\dots,F_{k,\lambda})] - E[ \phi(N_{1,\lambda},\dots,N_{k,\lambda})] \longrightarrow 0, \quad n\to \infty,
\]
where $(N_{1,\lambda},\dots,N_{k,\lambda})$ is a centered Gaussian vector with the same covariance matrix as $(F                _{1,\lambda},\dots,F_{k,\lambda})$.
\end{enumerate}

\end{thm}

\begin{rem}{\rm In the statement of Theorem \ref{t:main2}, we implicitly allow that the underlying Poisson measure $\eta$ also changes with $\lambda$. In particular, one can assume that the associated control measure $\mu = \mu_\lambda$ explicitly depends on $\lambda$. 
}
\end{rem}

\end{itemize}

\subsection{A general problem about marked point processes }

In this paper, we shall mainly deal with the normal approximation of functionals of marked Poisson point processes. Here is the general problem we are interested in. 

\begin{prob}
\label{prob:main}

Let $X$ be a compact subset of $\mathbb{R}^d$ endowed with the Lebesgue measure $\leb$, and let  $M$ be a locally compact space, that we shall call the \emph{mark space}, endowed with a probability measure $\nu$. We shall assume that $X$   contains the origin in its interior, is symmetric, that is: $X = -X$, and that the boundary of $X$ is negligible with respect to Lebesgue measure. We set $Z=X\times M$, and we endow such a product space with the measure $\mu=\leb\otimes \nu$ on $\mathbb{R}^d$. Let $\{\alpha_{\lambda}: \lambda>0\}$  be a collection of positive numbers. Define $\mu_{\lambda}=\lambda \mu$, and let $\eta_{\lambda}$ be a Poisson measure on $Z$ with control measure $\mu_{\lambda}$. For every $1 \leq i \leq k$ and every $\lambda>0$, let $h_{i,\lambda}$ be a fixed real function  belonging to $ L_{s}^2((\alpha_{\lambda}X)^i\times M^i)$.
Let $F_{\lambda}$ be defined as in (\ref{e:genseq}), where the multiple integrals are with respect to $\hat\eta_\lambda = \eta_\lambda - \lambda\mu$, and assume that each kernel $f_{i,\lambda}$ is of the form 
\begin{equation*}
f_{i,\lambda}(\bfx_{i})=\gamma_{i,\lambda}h_{i,\lambda}(\alpha_{\lambda}\bfx_{i}),\, \bfx_{i}\in Z^i,
\end{equation*}
for some $\gamma_{i,\lambda}>0$.
Which conditions on the $\gamma_{i,\lambda}$, $\alpha_{\lambda}$ and $h_{i,\lambda}$, $i=1,\dots,k$, yield the asymptotic normality of $$\tilde F_{\lambda}=\frac{F_{\lambda}-E[F_{\lambda}]}{\sqrt{\var(F_{\lambda})}}, \mbox{\ \ as } \lambda \to\infty ?$$

\end{prob}

\smallskip

Sufficient conditions for asymptotic normality, together with explicit estimates for the quantity $B_3$ appearing in Theorem \ref{t:main2}, are derived in Theorem \ref{th:general-clt}.

\begin{rem}
\begin{enumerate}

\item[ (a) ]In the above formulation of Problem \ref{prob:main}, introducing the scaling factor $\alpha_\lambda$ might seem redundant (since $h_{i,\lambda}$ also depends on $\lambda$). However, this representation is convenient for the applications developed below, where, for each $i=1,\dots,k$, the kernel $h_{i,\lambda}$ will be assumed to converge to some global function $h_{i}$ that does not depend on $\lambda$.

\item[(b)]  It is proved in \cite[Theorem 3.5.8]{SchWei} that any stationary  marked Poisson point process $\eta$ has intensity measure of the form  $\mu_{\lambda}=\lambda\leb\otimes \nu$, where $\lambda>0$, $\leb$ is the $d$-dimensional Lebesgue measure, and $\nu$ is a probability measure on $M$. In the parlance of stochastic geometry, a stationary marked point process $\eta=\{(t_{i},m_{i})\}$ is therefore always obtained from a \emph{ground process} $\eta_{0}=\{t_{i}\}$, that is, from a Poisson point process with intensity $\lambda \leb$, by attaching to each point $t_{i}$ an independent random mark $m_{i}$, drawn from $M$ according to the probability distribution $\nu$.  

\end{enumerate}

\end{rem}

\medskip

\subsection{Rescaled marked $U$-statistics}

Following \cite[Section 3.1]{lesmathias}, we now introduce the concept of a $U$-{\it statistic} associated with the Poisson measure $\eta$. This is the most natural example of an element of $L^2(P)$ having a finite \WI expansion.

\begin{defi} {\bf ($U$-statistics)} {\rm Fix $k\geq 1$. A random variable $F$ is called a $U$-{\it statistic of order} $k$, based on a Poisson measure $\eta$ with control $\mu$, if there exists a kernel $h\in L^1_s(\mu^k)$ such that
\begin{equation}\label{e:ustat}
F = \sum_{\bfx \in \eta^k_{\neq}} h(\bfx), 
\end{equation}
where the symbol $\eta^k_{\neq}$ indicates the class of all $k$-dimensional vectors $\bfx =(x_1,\dots,x_k)$ such that $x_i\in \eta$ and $x_i\neq x_j$ for every $1\leq i\neq j \leq k$. As made clear in {\rm \cite[Definition 3.1]{lesmathias}}, the possibly infinite sum appearing in (\ref{e:ustat}) must be regarded as the $L^1(P)$ limit of objects of the type $\sum_{\bfx\in \eta^k_{\neq}\cap A_n} f(\bfx)$, $n\geq 1$, where the sets $A_n \in Z^k$ are such that $\mu^k(A_n)<\infty$ and $A_n \uparrow Z^k$, as $n\to \infty$.}
\end{defi}

The following crucial fact is proved by Reitzner \& Schulte in \cite[Lemma 3.5 and Theorem 3.6]{lesmathias}:

\begin{prop}
\label{prop:L1L2}
Consider a kernel $h\in L_{s}^1(\mu^k)$ such that the corresponding $U$-statistic $F$ in (\ref{e:ustat}) is square-integrable. Then, $h$ is necessarily square-integrable, and $F$ admits a representation of the form (\ref{chaos}), with \begin{equation}\label{e:lastpenrose}
f_i(\bfx_{i}) := h_i(\bfx_{i})= \binom{k}{i}\int_{Z^{k-i}} h(\bfx_{i},\bfx_{k-i})\, d\mu^{k-i},\quad \bfx_{i}\in Z^i,
\end{equation}
for $1 \leq i \leq k$, and $f_{i}=0$ for $i>k$. In particular, $h=f_k$ and the projection $f_{i}$ is in $ L_{s}^{1,2}(\mu^i)$ for each $1 \leq i \leq k$.\\

\end{prop}

\begin{rem}{\rm
Somewhat counterintuitively, it is proved in \cite{lesmathias} that the condition $h\in L^{1,2}(Z^k)$ does not ensure that the associated $U$-statistic $F$ in (\ref{e:ustat}) is a square-integrable random variable.} 
\end{rem}

As discussed in the Introduction, $U$-statistics based on Poisson measures play a fundamental (and more or less explicit) role in many geometric problems -- like for instance those related with random graphs. The aim of this paper is to estimate as precisely as possible the contraction norms involved in Theorem \ref{t:main2} when $F$ has the form of a $U$-statistic whose kernel $h$ verifies some specific geometric assumptions. As anticipated, this significantly extends the analysis initiated in the second part of \cite{LacPec} -- where we only focussed on $U$-statistics of order $k=2$ with kernels equal to indicator functions. In the sequel, we will mostly assume that, as the intensity of the Poisson measure changes with $\lambda>0$, the underlying kernel $h_{\lambda}$ is the restriction of a fixed kernel $h$ whose argument is deformed by a factor $\alpha_{\lambda}>0$. The following problem will guide our discussion throughout Sections 4, 5 and 6: it can be regarded as a special case of Problem \ref{prob:main}.

\begin{prob}
\label{prob:Ustat}      
Let $Z$, $\mu_\lambda$, $\eta_\lambda$ and $\alpha_\lambda$, $\lambda>0$, be as in the statement of Problem \ref{prob:main}; in particular, $\eta_\lambda$ is a Poisson measure on the space $Z = X\times M$ (note that $Z$ is independent of $\lambda$). Fix $k\geq 1$ and let $h$ be a fixed real-valued function on $(\R^d)^k\times M^k$ whose restriction to $(\alpha_{\lambda}X)^k\times M^k$ belongs to $ L_{s}^1((\alpha_{\lambda}X)^k\times M^k)$ for every $\lambda>0$.  We are then interested in characterizing those collections of $U$-statistics of the type 
\begin{equation}
\label{eq:scaled-Ustat}
F_\lambda := F(h,X,M; \alpha_{\lambda},\eta_{\lambda})=\sum_{\bfx \in  \eta_{\lambda,\neq}^{k}}h(\alpha_{\lambda}\bfx),\quad \lambda>0,
\end{equation}
such that each $F_\lambda$ is square-integrable and moreover
\begin{equation*}
\tilde F(h,X,M; \alpha_{\lambda},\eta_{\lambda}):=\frac{F(h,X,M; \alpha_{\lambda},\eta_{\lambda})-E[F(h,X,M; \alpha_{\lambda},\eta_{\lambda})]}{\sqrt{\var(F(h,X,M; \alpha_{\lambda},\eta_{\lambda}))}}\stackrel{\rm Law}{\Longrightarrow} \norm,
\end{equation*}
as $\lambda\to \infty$.

\end{prob}

Sufficient conditions for the asymptotic normality of this type of \Uss, together with explicit estimates for the rates of convergence, will be derived in Theorem \ref{th:CLT-Ustat}.

\begin{examples}\label{ex:x}
\begin{description}
\item{(i)} Assume that either $\alpha_{\lambda}=1$ or $h$ is {\it homogeneous}, meaning  
\begin{equation}
\label{eq:homogeneous-kernel}
h(\alpha x)=\alpha^\beta h(x),\quad \forall\alpha>0,\quad \forall x\in \mathbb{R}^d,
\end{equation}
for some $\beta>0.$ Then the random variable $F_\lambda$ in (\ref{eq:scaled-Ustat}) is a \emph{geometric $U$-statistic}, where we use the terminology introduced in \cite{lesmathias}. Examples and sufficient conditions for normality are given in \cite{lesmathias}. Using Hoeffding decompositions, in Section \ref{sec:geom-Ustat} we shall provide an exhaustive characterization of their (central and non-central) asymptotic behavior. 

\item{(ii)} If $\alpha_{\lambda}=\lambda^{1/d}$, then the random variable $F_\lambda$ in (\ref{eq:scaled-Ustat}) verifies the identity in law
\begin{equation*}
F_\lambda \stackrel{\rm Law}{=}\sum_{\bfx_{k}\in (\alpha_{\lambda}X^k \cap \eta_{1,\neq}^{k})}h(\bfx_{k})
\end{equation*}
where $\eta_{1}$ is a homogeneous marked point process on $\mathbb{R}^d\times M$ with control measure $\leb\otimes \nu$.
\end{description}
\end{examples}

\section{Preliminary example: the importance of analytic \\ bounds in subgraph counting}\label{s:sub}

Before tackling Problem \ref{prob:main} and Problem \ref{prob:Ustat} in their full generality, we shall discuss a simple application of Theorem \ref{t:main2}-1 to CLTs associated with subgraph counting statistics in a standard disk graph model. The aim of this section is to familiarize the reader -- in a more elementary setting -- with some of the computations developed in the remainder of the paper. In particular, the forthcoming Theorem \ref{t:sub} involves $U$-statistics that are based on {\it stationary kernels}, a notion that will be formally introduced in Section \ref{sec:invariant} in the general framework of marked point processes. One important message that we try to deliver is that the bound (\ref{e:bnd}) allows one to deal at once with all possible asymptotic regimes of the disk graph model -- as defined at the end of the forthcoming Section \ref{ss:subframe}. The idea that analytic bounds such as (\ref{e:bnd}) may be used to simultaneously encompass a wide array of probabilistic structures is indeed one of the staples of present paper and \cite{LacPec}.

\begin{rem}{\rm
The CLT presented in Section \ref{ss:submain} is both a special case and a strengthening of the findings contained in \cite[Section 3.4]{penrosebook}. Indeed, in such a reference the author deduces multidimensional versions of the CLT below, but without providing explicit bounds in the Wasserstein distance. Also, when $k=2$ (that is, edge counting) the results of this section in the case of a uniform density $f$ (comprising the upper bounds) are a special case of \cite[Theorem 4.9]{LacPec}. The notation used below has been chosen in order to be loosely consistent with the one adopted in reference \cite{penrosebook}.
}
\end{rem}

\subsection{Framework}\label{ss:subframe}

We fix $d\geq 1$, as well as a bounded and almost everywhere continuous probability density $f$ on $\R^d$. We denote by $Y = \{Y_i : i\geq 1\}$ a sequence of $\R^d$-valued i.i.d. random variables, distributed according to the density $f$. For every $n=1,2,...$, we write $N(n)$ to indicate a Poisson random variable with mean $n$, independent of $Y$. It is a standard result that the random measure $\eta_n = \sum_{i=1}^{N(n)} \delta_{Y_i}$, where $\delta_x$ indicates a Dirac mass at $x$, is a Poisson measure on $\R^d$ with control $\mu_n(dx) = n f(x)dx$ (where $dx$ stands for the Lebesgue measure). We shall also write $\hat{\eta}_n = \eta_n - \mu_n$, $n\geq 1$.

\medskip

 Let $\{t_n : n\geq 1\}$ be a sequence of strictly decreasing positive numbers such that $\lim_{n\to \infty} t_n = 0$. For every $n$, we define $G'(Y;t_n)$ to be the {\it disk graph} obtained as follows: the vertices of $G'(Y;t_n)$ are given by the random set $V_n = \{Y_1,...,Y_{N(n)}\}$ and two vertices $Y_i,Y_j$ are connected by an edge if and only if $\| Y_i - Y_j \|_{\R^d} \in (0,t_n)$. By convention, we set $G'(Y;t_n) = \emptyset$ whenever $N(n)=0$. Now fix $k\geq 2$, and let $\Gamma $ be a connected graph of order $k$. For every $n\geq 1$, we shall denote by $G'_n(\Gamma)$ the number of induced subgraphs of $G'(Y;t_n)$ that are isomorphic to $\Gamma$, that is: $G'_n(\Gamma)$ counts the number of subsets $\{i_1,...,i_k\} \subset \{1,...,N(n)\}$ such that the restriction of $G'(Y;t_n)$ to $\{Y_{i_1},...,Y_{i_k}\}$ is isomorphic to $\Gamma$. We shall assume throughout the following that $\Gamma$ is {\it feasible} for every $n$. This requirement means that the probability that the restriction of $G'(Y;t_n)$ to $\{Y_1,...,Y_k\}$ is isomorphic to $\Gamma$ is strictly positive for every $n$.
 
\medskip 

We are interested in studying the Gaussian fluctuations, as $n\to\infty$, of the random variable $G'_n(\Gamma)$. In order to do this, one usually distinguishes the following three regimes.

\begin{enumerate}

\item[\bf (R1)] $nt_n^d \to 0$ and $n^k(t_n^d)^{k-1} \to \infty$;

\item[\bf (R2)] $nt_n^d \to \infty$;

\item[\bf (R3)]  {\it (Thermodynamic regime)} $nt_n^d \to c$, for some constant $c\in (0,\infty)$.

\end{enumerate}

The following statement collects some important estimates from \cite[Chapter 3]{penrosebook}. Given positive sequences $a_n,b_n$, we shall use the standard notation $a_n\sim b_n$  to indicate that, as $n\to \infty$, $a_n/b_n \to 1$.

\begin{prop} \label{p:penrose}
Under the three regimes {\rm {\bf (R1)}, {\bf (R2)}} and {\rm {\bf (R3)}}, one has that $E[G'_n(\Gamma)]\sim n^{k-1}(t_n^d)^{k-1}$. Moreover there exists strictly positive constants $c_1,\,c_2,\, c_3$ such that, as $n\to\infty$,
\begin{itemize}

\item[--] under {\rm {\bf (R1)}}, ${\rm Var}(G'_n(\Gamma))\sim c_1\times n^k(t_n^d)^{k-1} ;$ 

\item[--] under {\rm {\bf (R2)}}, ${\rm Var}(G'_n(\Gamma))\sim c_2\times n^{2k-1}(t_n^d)^{2k-2} ;$

\item[--] under {\rm {\bf (R3)}}, ${\rm Var}(G'_n(\Gamma))\sim c_3\times n$ .

\end{itemize}

\end{prop}

The next subsection contains the announced normal approximation result.

\subsection{The CLT} \label{ss:submain}

For every $n\geq 1$, we set
$$
\tilde{G}'_n(\Gamma) = \frac{{G}'_n(\Gamma) - E[{G}'_n(\Gamma)] }{{\rm Var}({G}'_n(\Gamma))^{1/2}},
$$
and we consider a random variable $N\sim \mathscr{N}(0,1)$. The following statement is the main achievement of the present section.

\begin{thm}\label{t:sub} Let the assumptions and notation of this section prevail. There exists a constant $C>0$, independent of $n$, such that, for every $n\geq 1$,

\begin{itemize}

\item[--] under {\rm {\bf (R1)},}   $d_W(\tilde{G}'_n(\Gamma) , N) \leq C \times (n^k (t_n^d)^{k-1})^{-1/2}  ; $ 

\item[--] under {\rm {\bf (R2)}--{\bf (R3)},} $ d_W(\tilde{G}'_n(\Gamma) , N) \leq C \times n ^{-1/2}  . $

\end{itemize}

\noindent In particular, under the three regimes {\rm {\bf (R1)}, {\bf (R2)}} and {\rm {\bf (R3)}}, one has that $\tilde{G}'_n(\Gamma)$ converges in distribution to $N$, as $n\to \infty$.
\end{thm}
\noindent {\it Proof:} By construction, the random variable $G'_n(\Gamma)$ has the form
\[
G'_n(\Gamma) = \sum_{(x_1,...,x_k) \in \eta_{n,\neq}^k} h_{\Gamma,t_n}(x_1,...,x_k),
\]
where the function $h_{\Gamma,t_n} : \R^m \to \R$ equals $1/k!$ if the restriction of $G'(Y ;t_n)$ to $\{x_1,...,x_k\}$ is isomorphic to $\Gamma$, and equals 0 otherwise. Plainly, $h_{\Gamma,t_n}$ has the following two characteristics: (i) $h_{\Gamma,t_n}$ is symmetric, and (ii) $h_{\Gamma,t_n}$ is {\it stationary}, in the sense that it only depends on the norms $\| x_i-x_j\|_{\R^d}$, $i\neq j$. Using \cite[Lemma 3.5]{lesmathias}, we deduce that the random variable $G'_n(\Gamma)$ (that has trivially moments of all orders) admits the following chaotic decomposition
\[
G'_n(\Gamma) = E[G'_n(\Gamma)] + \sum_{i=1}^k I_i^{\hat{\eta}_n}(h_i),
\]
where $E[G'_n(\Gamma)] = \int_{(\R^d)^k}  h_{\Gamma,t_n}d\mu_n^k$, 
\[
h_i(x_1,...,x_i) = \binom{k}{i} \int_{(\R^d)^{k-i}} \!\!\!\!\!\!\!h_{\Gamma,t_n}(x_1,...,x_i,y_1,...,y_{k-i}) \mu_n^{k-i}(dy_1,...,dy_{k-i}) :=\binom{k}{i}h^{(i)}_{\Gamma,t_n}(x_1,...,x_i),
\]
and $I_i^{\hat{\eta}_n}$ indicates a (multiple) Wiener-It\^o integral of order $i$, with respect to the compensated measure $\hat{\eta}_n$. Writing $v^2_n := {\rm Var}(G'_n(\Gamma))$, by virtue of (\ref{e:bnd}) one has that the result is proved once it is shown that, for every $j=1,...,k$ and for every $1\leq l\leq r\leq i\leq j\leq k$ such that $l\neq j$, the following holds as $n\to\infty$: under {\bf (R1)},
\begin{eqnarray}\label{e:p1}
&&\! \! \! \! \! \! \! \! \! \frac{\| h_{\Gamma,t_n}^{(j)}\|^4_{L^4(\mu^i_n)}}{v_n^4} =O\left( \frac{1}{n^{k}(t_n^d)^{(k-1)}}\right),   \, \frac{\| h_{\Gamma,t_n}^{(i)}\star_r^l h_{\Gamma,t_n}^{(j)}\|^2_{L^2(\mu^{i+j-r-l}_n)}}{v_n^4} = O\left( \frac{1}{n^{k}(t_n^d)^{(k-1)}}\right)\!\!,
\end{eqnarray}
and, under {\bf (R2)--(R3)},
\begin{eqnarray}\label{e:p2}
&&\! \! \! \! \! \! \! \! \! \frac{\| h_{\Gamma,t_n}^{(j)}\|^4_{L^4(\mu^i_n)}}{v_n^4} = O(n^{-1}),   \quad \frac{\| h_{\Gamma,t_n}^{(i)}\star_r^l h_{\Gamma,t_n}^{(j)}\|^2_{L^2(\mu^{i+j-r-l}_n)}}{v_n^4} = O(n^{-1}).
\end{eqnarray}
Observe that $\| h_{\Gamma,t_n}^{(j)}\|^4_{L^4(\mu^i_n)} = \| h_{\Gamma,t_n}^{(j)}\star_j^0 h_{\Gamma,t_n}^{(j)} \|^2_{L^2(\mu^i_n)}$, so that we just have to check the second relation in (\ref{e:p1}) and in (\ref{e:p2}) for every $(i,j,r,l)$ in the set $Q = \{1\leq l\leq r\leq i\leq j\leq k, \, l\neq j\}\cup\{ i=r=j,\, l=0\}$. For every $(i,j,r,l)\in Q$ we define the function $h_{\Gamma, t_n}^{(i,j,r,l)} : (\R^d)^{\alpha} \to \R$, where $\alpha = \alpha(i,j,r,l) = 4k-i-j-r+l$, as follows:
\begin{eqnarray}
h_{\Gamma, t_n}^{(i,j,r,l)}(x_1,...,x_\alpha) &=& h_{\Gamma,t_n}({\bf x}^{(1)}_{k-i}, {\bf x}^{(2)}_{i-r}, {\bf x}^{(3)}_{r-l}, {\bf x}^{(4)}_{l} )h_{\Gamma,t_n}({\bf x}^{(5)}_{k-j}, {\bf x}^{(6)}_{j-r}, {\bf x}^{(3)}_{r-l}, {\bf x}^{(4)}_{l} )\times \\
&& \times h_{\Gamma,t_n}({\bf x}^{(7)}_{k-i}, {\bf x}^{(2)}_{i-r}, {\bf x}^{(3)}_{r-l}, {\bf x}^{(8)}_{l} )h_{\Gamma,t_n}({\bf x}^{(9)}_{k-j}, {\bf x}^{(6)}_{j-r}, {\bf x}^{(3)}_{r-l}, {\bf x}^{(8)}_{l} ),
\end{eqnarray}
where the numbered bold letters stand for packets of variables providing a lexicographic decomposition of $(x_1,...,x_\alpha)$, for instance ${\bf x}^{(1)}_{k-i} = (x_1,...,x_{k-i})$, ${\bf x}^{(2)}_{i-r} = (x_{k-i+1},...,x_{k-r})$, and so on. In this way, one has that $({\bf x}^{(1)}_{k-i}, {\bf x}^{(2)}_{i-r}, {\bf x}^{(3)}_{r-l}, {\bf x}^{(4)}_{l}, {\bf x}^{(5)}_{k-j}, {\bf x}^{(6)}_{j-r} ,{\bf x}^{(7)}_{k-i}, {\bf x}^{(8)}_{l} , {\bf x}^{(9)}_{k-j} )= (x_1,...,x_\alpha) $, and we set ${\bf x}^{(a)}_{p}$ equal to the empty set whenever $p=0$. Observe that each function $h_{\Gamma, t_n}^{(i,j,r,l)}$ takes values in the set $\{0, k!^{-4}\}$; moreover, the connectedness of the graph $\Gamma$ implies that the mapping $(x_2,...,x_\alpha) \mapsto h_{\Gamma, 1}^{(i,j,r,l)}(0,x_2,...,x_\alpha)$, where $0$ stands for the origin, has compact support. Applying (\ref{eq:norm-contraction}), one has that, for every $(i,j,r,l)\in Q$,
\[
\| h_{\Gamma,t_n}^{(i)}\star_r^l h_{\Gamma,t_n}^{(j)}\|^2_{L^2(\mu^{i+j-r-l}_n)} = n^{\alpha}\int_{(\R^d)^{\alpha}} h_{\Gamma, t_n}^{(i,j,r,l)}(x_1,...,x_\alpha) f(x_1)\cdots f(x_\alpha)dx_1\cdots dx_\alpha,
\] 
where $\alpha = 4k-i-j-r+l$, as before. Applying the change of variables $x_1 = x$ and $x_i = t_ny_i +x$, for $i=2,...,\alpha$, one sees that
\begin{eqnarray*}
&& \| h_{\Gamma,t_n}^{(i)}\star_r^l h_{\Gamma,t_n}^{(j)}\|^2_{L^2(\mu^{i+j-r-l}_n)} \\
&&= n^{\alpha}(t_n^d)^{\alpha-1} \int_{\R^d}f(x)  \int_{(\R^d)^{\alpha -1}} h_{\Gamma, 1}^{(i,j,r,l)}(0,y_2,...,y_\alpha) f(x+t_n y_2)\cdots f(x+t_ny_\alpha)dxdy_2\cdots dy_\alpha.
\end{eqnarray*}
Since, by dominated convergence, the integral on the RHS in the previous equation converges to the constant
\[
\int_{\R^d}f^\alpha (x)dx  \int_{(\R^d)^{\alpha -1}} h_{\Gamma, 1}^{(i,j,r,l)}(0,y_2,...,y_\alpha)dy_2\cdots dy_\alpha,
\]
we deduce that $\| h_{\Gamma,t_n}^{(i)}\star_r^l h_{\Gamma,t_n}^{(j)}\|^2_{L^2(\mu^{i+j-r-l}_n)} = O( n^{\alpha}(t_n^d)^{\alpha-1})$ for every $(i,j,r,l)\in Q$. Using this estimate together with Proposition \ref{p:penrose}, and after some standard computations, one sees that relations (\ref{e:p1})--(\ref{e:p2}) are in order, and the desired conclusion is therefore achieved.
\qed

\begin{rem}{\rm

\begin{enumerate}

\item The CLT under the regime {\bf (R1)} could in principle be deduced from Theorem 3.4 in \cite{penrosebook}, via a Poissonization argument. However, since this strategy is based on an intermediate Poisson approximation, in this way one would obtain suboptimal rates of convergence (in the Kolmogorov distance).

\item It is interesting to note that our proof of Theorem \ref{t:sub} is based on exactly the same change of variables that one usually applies in variance and covariance estimates -- see e.g. \cite[Section 3.3]{penrosebook}.

\end{enumerate}

}
\end{rem}

As anticipated, in what follows we will show that the kind of arguments displayed in the previous proof can be extended to the framework of functionals of marked point processes.

\section{Technical estimates on rescaled contractions}\label{s:technical}

\subsection{Framework and general estimates}

In this section, we collect several analytic estimates on the norms contractions of multivariate functions satisfying some specific geometric assumption. They will be used in further sections to study the asymptotic behavior of $U$-statistics.


\begin{rem}{\bf (Some conventions)}\label{r:conventions}

\begin{enumerate}

\item[(a)] Unless otherwise specified, throughout this section $Z=X\times M$, $\mu_\lambda$ and $\alpha_\lambda$, $\lambda>0$, are defined as in the statement of Problem \ref{prob:main}. In particular, $X$ is a symmetric compact subset of $\R^d$, for some fixed integer $d\geq 1$, with $0$ in its interior and negligible boundary. We shall also use the shorthand notation $Z_{\lambda}=\alpha_{\lambda}Z = (\alpha_\lambda X)\times M$, and $\Zi^{k}=(\mathbb{R}^d\times M)^k.$

\item[(b)] A point $x$ in $Z=X\times M$ is represented as $x=(t,m)$, where $t\in X$ is the spatial variable and $m\in M$ is the mark. A $k$-tuple $(x_{1},\dots,x_{k}) \in Z_\infty^k$ is denoted by the bold letter $\bfx_{k}$; a $k$-tuple $(t_{1},\dots,t_{k}) \in (\R^d)^k$ is denoted by the bold letter $\bft_{k}$. When considering symmetric functions $f(x_{1},\dots,x_{k})$ on $Z^k$, the chosen order of the $x_{i}$ in the argument of $f$ is immaterial. We will sometimes separate the set of variables $\bfx_{k}=(\bfx_{j},\bfx_{k-j})$, where $\bfx_{j}$ is the $j$-tuple formed by the $j$ first variables and $\bfx_{k-j}$ the $k-j$ last variables. By symmetry, one has of course that $f(\bfx_{k})=f(\bfx_{j},\bfx_{k-j})=f(\bfx_{k-j},\bfx_{j})$.

\item[(c)] In our framework, any geometric operation $\theta$ applied to a point $x=(t,m)$ is by definition only applied to the spatial coordinate. For instance, the $s$-translation is given by $\theta_{s}x=x+s=(s+t,m)$, whereas the $\alpha$-dilatation ($\alpha \geq 0$) is $\delta_{\alpha}x=\alpha x=(\alpha t,m)$. 
These conventions are canonically extended to $k$-tuples of points, as well as to subsets of $\Zi^{k}$. Observe that, if $t\in \R^d$ and ${\bf t}_k = (t_1,...,t_k)\in X^k$, then ${\bf t}_k +t = (t_1+t,t_2+t,...,t_k+t)$.

\item[(d)] Let $A\subset \R^d$. For every $k\geq 2$ there exists a canonical bijection $\psi_k$ between the two sets $(A \times M)^k$ and $A^k\times M^k$, given by 
\[
\psi_k : ((t_1,m_1),...,(t_k,m_k)) \mapsto (t_1,...,t_k,m_1,...,m_k), \quad t_i\in A, \,\, m_i\in M.
\]
To simplify our discussion, in what follows {\it we will systematically identify the two sets} $(A \times M)^k$ and $A^k\times M^k$ by implicitly applying the mapping $\psi_k$ and its inverse. For instance: $({\bf t}_k, {\bf m}_k) \in (A \times M)^k$, where ${\bf t}_k = (t_1,...,t_k) \in (\R^d)^k$ and ${\bf m}_k = (m_1,...,m_k)\in M^k$ is shorthand for $\psi_k^{-1}({\bf t}_k, {\bf m}_k) \in (A \times M)^k$; writing ${\bf x}_k = ({\bf t}_k, {\bf m}_k)$, where ${\bf x}_k = (x_1,...,x_k)$, $x_i\in \R^d \times M$, means indeed that $({\bf t}_k, {\bf m}_k) = \psi_k ({\bf x}_k)$; if a function $f$ is defined on some subset of $(\R^d \times M)^k$, we write $f({\bf t}_k,{\bf m}_k)$ to indicate the quantity $f(\psi_k^{-1}({\bf t}_k,{\bf m}_k))$ (and an analogous convention holds for functions defined on a subset of $(\R^d)^k \times M^k$).

\end{enumerate}
\end{rem}

The following statement shows how the norms of contractions are modified by a rescaling of the underlying kernels.

\begin{prop}

Let $\alpha,\gamma,\gamma'>0$, $h\in L^2(\alpha Z ^k; \mu ^k)$, $h'\in L^2(\alpha Z^q; \mu^q)$, where $1 \leq q\leq k$.
Define 
\begin{align*}
f(\bfx_{k})=\gamma h(\alpha\bfx_{k}),\,\bfx_{k}\in Z ^k,\\
f'(\bfx_{q})=\gamma'h'(\alpha\bfx_{q}),\, \bfx_{q}\in Z ^q.
\end{align*}
Fix $1\leq l \leq r \leq q \leq  k $, 
and set $m=q+k-r-l$. For every $\lambda>0$ one has that
\begin{equation}
\label{eq:contraction-scaling}
\|f\star_{r}^l f'\|_{L^2(Z^{m}; \mu_{\lambda}^{m})}^2=\gamma^2(\gamma')^2 (\lambda \alpha^{-d})^{m+2l} \|h\star_{r}^l h'\|_{L^2(\alpha Z^m; \mu^m)}^2,
\end{equation}
(the contractions $f\star_{r}^l f'$ and $h\star_{r}^l h'$ being realized, respectively, via $\mu_\lambda$ and $\mu$) and for $p\geq 1$,
\begin{equation}
\label{eq:Lp-scaling}
\|f\|^p_{L^p(Z ^k;\mu_{\lambda} ^k)}=\gamma^p(\lambda\alpha^{-d})^k\|h\|_{L^p(\alpha Z ^k;\mu ^k)}^p.
\end{equation}

\end{prop}

\begin{proof}
The proof relies on a change of variables in (\ref{eq:norm-contraction}) where all variables are multiplied by a factor $\alpha$. 
\begin{align*}
&\|f\star_{r}^l f'\|_{L^2(Z^{m}; \mu_{\lambda}^{m})}^2\\
&= \gamma^2(\gamma')^2\lambda^{2l} \int_{Z^{m}}\lambda^{m}d\mu^{m} \int_{Z^{2l}}h(\alpha(\bfx_{k-r},\bfy_{r-l},\bfz_{l}))h(\alpha(\bfx_{k-r},\bfy_{r-l},\bfz'_{l}))\\
&\hspace{5cm}h'(\alpha(\bfx'_{q-r},\bfy_{r-l},\bfz_{l}))h'(\alpha(\bfx'_{q-r},\bfy_{r-l},\bfz'_{l}))d\mu^{2l}\\
&= \gamma^2(\gamma')^2 \lambda^{m+2l}\int_{\alpha Z^{m+2l}}\alpha^{-d(m+2l)}d\mu^{m+2l}h(\bfx_{k-r},\bfy_{r-l},\bfz_{l})h(\bfx_{k-r},\bfy_{r-l},\bfz'_{l})\\
&\hspace{5cm}h'(\bfx'_{q-r},\bfy_{r-l},\bfz_{l})h'(\bfx'_{q-r},\bfy_{r-l},\bfz'_{l})\\
&=\gamma^2(\gamma')^2 (\lambda \alpha^{-d})^{m+2l} \|h\star_{r}^l h'\|_{L^2(\alpha Z^m; \mu^m)}^2.
\end{align*}
Formula (\ref{eq:Lp-scaling}) is proved by the same route.

\end{proof}



\subsection{Stationary kernels}
\label{sec:invariant}


The estimates of this section will be used to study the RHS of (\ref{eq:contraction-scaling}) when $\alpha=\alpha_{\lambda}$ depends on $\lambda$ and $\alpha_\lambda\to \infty$, assuming that $h$ and $h'$ are invariant in the sense described below. Note that, if $\alpha_{\lambda}\to \infty$ and since $X$ has $0$ in its  interior, with our notation one has that 
\begin{equation*}
\Zi=\cup_{\lambda>0}\alpha_{\lambda}Z=\cup_{\lambda>0}Z_{\lambda}=\mathbb{R}^d\times M.
\end{equation*} 
  
\medskip  
  
\noindent We say that a function $h$ defined on $\Zi^k$ is \emph{invariant under translations}, or \emph{stationary}, if
\begin{equation*}
h(\bft_{k},\bfm_{k})=h(\bft_{k}+t,\bfm_{k})
\end{equation*}
for every $t$ in $\mathbb{R}^d$ and every $(\bft_{k},\bfm_{k})\in Z_\infty^k$. This property implies that, if ${\bf t}_ k = (t_1,...,t_k)$,
\begin{equation}
\label{eq:fact}
h(\bft_{k},\bfm_{k})=h(0,\bft_{k-1}-t_{1},\bfm_{k})=\fact  h(\bft_{k-1}-t_{1},\bfm_{k}),
\end{equation}
where $\fact h : (\R^d)^{k-1} \times M^k \to \R$ is given by $\fact  h({\bf s}_{k-1},\bfm_{k}) = h(0,\bfs_{k-1},\bfm_{k})$ and, according to our conventions, $\bft_{k-1} = (t_2,...,t_k)$. As usual, we identify the space $(\R^d)^0$ withe a one-point set: this is consistent with the fact that a function on $\R^d$ is stationary if and only if it is constant. Note that, in the previous formula (\ref{eq:fact}), the choice of the variable $t_{1}$ among the variables $t_{i}$, $i=1,...,k$, is immaterial whenever $h$ is symmetric in its $k$ variables. Given ${\bf x}_k = (\bft_k, \bfm_k) \in Z^k_\infty$ and $t\in \R^d$ (that is, $t$ is a spatial variable), we shall write, by a slight abuse of notation, $\bfx_k +t = (\bft_k +t , \bfm_k)$: for instance, with this notation a function $h$ on $Z_\infty^k $ is stationary if and only if $h(\bfx_k ) = h(\bfx_k+t )$ for every $t\in\R^d$ and every $\bfx_k \in Z^k_\infty$.

\begin{example}{\rm Assume $M=\emptyset$. Then, a stationary symmetric kernel is given by
$$
h(t_1,...,t_k) = g( \| t_i - t_j \|_{\R^d} : 1\leq i<j\leq k),
$$
where $g : \R^{k(k-1)/2}_+\to \R $ is some symmetric mapping. For instance, if $g(a_1, ..., a_{k(k-1)}) = \prod_{j}{\bf 1}_{a_j \leq \delta}$, then $h(t_1,...,t_k)$ equals 1 or 0 according as the distance between every pair of coordinates of the vector $(t_1,...,t_k)$ does not exceed $\delta$.   Kernels of this type appear in Section \ref{s:sub}, in the framework of subgraph counting, as well as in Section \ref{s:applications}, where we deal with random radii of interaction.
}
\end{example}

\begin{rem}

For technical purposes we also define, for $X,X' \subseteq \R^d$,
\begin{equation*}
X+X'=\{t+t' :  t\in X,\, t'\in X'\}, \quad \hat X=X-X=X+(-X)
\end{equation*}
and let $\check X$ be the largest symmetric set such that $\check X-\check X\subseteq X$. Remark that if $t_{j}\in X, 1 \leq j \leq k$, then $t_{j}-t_{1}\in \hat X$, for every $1 \leq j \leq k$. Also, if each $t_{j}\in \check X$ ($j\geq 1$), then one can exploit the symmetry of $X$ and deduce that $(t_{1},...,t_k)=(s_1,s_{2}-s_{1},...,s_k-s_1)$, for some $\bfs_{k}=(s_{1},\dots,s_{k})\in  X^k$. It follows that if $\varphi$ denotes the change of variables $\varphi(\bft_{k})=(t_{1}, t_{2}-t_{1},\dots,t_{k}-t_{1})$, then one has that $ \check X^{k} \subseteq \varphi(X^k) \subseteq X\times \hat X^{k-1}$ (this will be useful in further change of variables). We also set $ \hat Z^q=\hat X^q\times M^{q}$ and $ \check Z^q=\check X^q\times M^{q}$.
 
\end{rem}

\begin{rem}
\begin{enumerate}
\item[(i)]
In (\ref{eq:fact}), 
$\fact h$ is a function whose argument has $k-1$ spatial variables and $k$ mark variables. To deal with this situation, we use the underlined symbol $\mrk \bfx_{k-1}=(\bft_{k-1},\bfm_{k})$ to indicate vectors such that the number of mark variables is one plus the number of spatial variables.
Accordingly, we write $\mrk{  Z^q}=X^q\times M^{q+1}, \,q\geq 0$, to indicate the collection of all vectors mark variable in the product space. Write also $\mrk{ \hat Z^q}=\hat X^q\times M^{q+1}, q\geq 0$, $\mrk{\Zi^{q}}=(\mathbb{R}^d)^q\times M^{q+1}, \mrk \mu^q=\leb^q\otimes \nu^{q+1}.$
At first reading, one can consider the simple framework where $M=\emptyset$, in which case we have simply $\mrk \bfx_{q}=\bfx_{q}=\bft_{q}, \mrk { Z^q}= X^q, q\geq 1$. 

\item[(ii)] Let $A\subset \R^d$. Analogously to the convention introduced in Remark \ref{r:conventions}-(d), in what follows we shall identify the sets $A^q\times M^{q+1}$ and $(A\times M)^i \times M\times (A\times M)^{q-i}$ by tacitly applying the canonical bijection between them (the chosen index $i\in \{1,...,q\}$ will be always clear from the context). 

\end{enumerate}
\end{rem}


Let us introduce some further notation, which is required in order to express our main bounds.
For the rest of the section we fix a probability density $\kappa(t)$ on $\mathbb{R}^d$, such that $0<\kappa\leq 1$. 
\begin{rem}{\rm The estimates proved in this section continue to hold (up to some multiplicative constant) whenever the density $\kappa$ is bounded from above by some constant $M>0$. The value $M=1$ has been chosen in order to simplify some of the formulae to follow. Note that, in the applications developed in Section \ref{s:applications}, the role of the upper bound for $M$ is sometimes taken to be different from 1.
}
\end{rem}

\medskip

For $s\geq 1$ and $ \bft_{s}=(t_{1},\dots,t_{s})\in (\mathbb{R}^d)^s$, we set
\begin{equation*}
\kappa_{s}(\bft_{s})=\kappa(t_{1})\dots \kappa(t_{s}).
\end{equation*} 
Plainly, $\kappa_{s}$ is a probability density on $(\mathbb{R}^d)^s$, satisfying $0 <   \kappa_{s+1}(\bft_{s+1}) \leq \kappa_{s}(\bft_{s})\leq 1$ for every $\bft_{s}\in (\mathbb{R}^d)^s$ and $\bft_{s+1}=(\bft_{s},t)$, with $t\in \mathbb{R}^d$.
For a measurable non-negative function $\fact h$ on $\mrk{ \Zi^{k-1}}, k \geq 1$, define for $p=2,4$
\begin{equation}\label{e:akappa}
A_{\kappa,p}(\fact h)=\int_{\mrk{\Zi^{k-1}}}\kappa_{k-1}(\bft_{k-1})^{-1}\fact h^p(\underline{\bfx}_{k-1})d\mrk \mu^{k-1}
\end{equation}
where $\bft_{k-1}$ stands for the spatial variables of $\mrk\bfx_{k-1}=(\bft_{k-1},\bfm_{k})$, $\bfm_{k}\in M^{k}$. We will simply write $A_{p}$, whenever the density $\kappa$ is unambiguously defined (as it is the case, for the rest of the section). The following lemma is fundamental in this article, it gives estimates for the norms and the contractions of stationary functions restricted to growing bounded domains.

\begin{rem}
Clearly, the quantities $A_{\kappa,2}(\fact h)$ and $A_{\kappa,4}(\fact h)$ defined above can be infinite. As shown in the applications developed later in the paper, the subtle point is, given a kernel $\fact h$, to find a density $\kappa$ such that $A_{\kappa,2}(\fact h), \, A_{\kappa,4}(\fact h) < \infty$, in such way that the bounds appearing in the forthcoming Lemma \ref{lm:inv-kernel} are finite. Often, $\kappa(x)=C(1+\|x\|)^\alpha$ for $C>0, \alpha<0$ well chosen, but if $\fact h$ presents a specific anisotropic behaviour, a more adapted density can be chosen.
\end{rem}

The following notation is borrowed from \cite{LacPec}.

\begin{rem}{\bf (Asymptotic equivalence notation) }{\rm
Given two mappings $\lambda \mapsto \gamma_{\lambda},\,  \lambda \mapsto \delta_{\lambda}$, we write $\gamma_{\lambda}\asymp \delta_{\lambda}$ if there are two constants $C,C'>0$ such that $C \gamma_{\lambda} \leq \delta_{\lambda} \leq C' \gamma_{\lambda}$ for $\lambda$ sufficiently large. We write $\gamma_{\lambda}\sim\delta_{\lambda}$ if $\delta_{\lambda}>0$ for $\lambda$ sufficiently large and $\alpha_{\lambda}/\delta_{\lambda}\to 1$.}  
\end{rem}


\begin{lemma}
\label{lm:inv-kernel}
Let $h, g$ be  symmetric non-negative stationary functions on $\Zi^{k} $ and $\Zi^{q}$ respectively, $1 \leq q \leq k$. Let $\fact h, \fact g$ be the kernels defined by (\ref{eq:fact}), and assume that $A_{\kappa, p}(\fact h) =A_{p}(\fact h) <\infty$ and $A_{\kappa, p}(\fact g) =A_{p}(\fact h) <\infty$ for some probability density $\kappa\in (0,1]$. Then, for $1 \leq l \leq r \leq q \leq k$ such that $l \leq k-1$, setting $m=k+q-r-l$, 
\begin{align}
\label{eq:contraction-norms}
\|h\star_{r}^r g\|^2_{L^2(Z^{m};\mu^m)}  \leq \|h\star_{r}^r g\|^2_{L^2(Z\times Z_{\infty}^{m-1};\mu^m)} &\leq  \leb(X) \sqrt{A_{2}(\fact h)A_{4}(\fact h)A_{2}(\fact g)A_{4}(\fact g)},\quad\quad\text{ if $r=l$},\\
\nonumber \|h\star_{r}^l g \|^2_{L^2(Z^m;\mu^m)}  \leq \|h\star_{r}^l g \|^2_{L^2(Z\times Z_{\infty}^{m-1};\mu^m)} &\leq \leb(X)\sqrt{A_{4}(\fact g)A_{4}(\fact h)},\quad\quad \text{ if }r-l>0,
\end{align} 
and for $p =2,4,$ 
\begin{equation}
\label{eq:contraction-norms2}
\|h\|_{L^p(\lambda Z^k)}^p\sim \lambda^d\leb(X)\|\fact{h}\|_{L^p(\Zi^k)}^p\quad\text{ as $\lambda\to \infty$}.
\end{equation}
Note that the contractions appearing in formula (\ref{eq:contraction-norms}) are realized with respect to the measure $\mu$.
Concerning the middle terms of the two inequalities (\ref{eq:contraction-norms}), the symbol $Z\times Z_{\infty}^{m-1}$ has to be interpreted as follows: in the norm involving the contraction $\star^r_r$ one of the variables in the argument of $h$ is integrated over $Z$, while the others are integrated over $Z_{\infty}$; in the norm involving the contraction $\star^l_r$ one of the variables in common between $h$ and $g$ is integrated over $Z$, while the others are integrated over $Z_{\infty}$. A similar convention holds for the middle term of inequality (\ref{eq:contraction-norms2}).

\end{lemma}




\noindent {\it Proof:} 
Let us start by proving   (\ref{eq:contraction-norms2}).
\begin{align*}
\|h\|_{L^p (\lambda Z^k)}^p&=\int_{\lambda Z^k}h(\bfx_{k})^pdx =\int_{\lambda X\times \lambda\mrk{Z^{k-1}}}\fact{h}(\mrk{\bfx_{k-1}}-t_{1})^pdt_{1}d\mrk{\bfx_{k-1}}\\
&=  \lambda^d\int_{ X\times \lambda\mrk{Z^{k-1}}}\fact{h}(\mrk{\bfx_{k-1}}-\lambda t_{1})^pdt_{1}d\mrk{\bfx_{k-1}}.
\end{align*}
For  every $t_{1}\in X$ not on the boundary (i.e. almost every $t_{1}$ for Lebesgue measure), since $X$ is symmetric, $X+t_{1}$ contains $0$ in its interior, whence $1_{\lambda(X+t_{1})}$ converges $\leb$-a.e. to $1_{\mathbb{R}^d}$. It yields by Lebesgue theorem on $X$ that
\begin{equation*}
\|h\|_{L^p (\lambda Z^k)}^p\sim \lambda^d \leb(X)\int_{\Zi^k}\fact{h}(\mrk{\bfx_{k-1}})^p\mrk{\bfx_{k-1}}
\end{equation*}
  with the domination
\begin{equation*}
\int_{\lambda Z^k}\mrk{h}(\mrk{\bfx_{k-1}}-\lambda^{-1}t_{1})^pd\mrk{\bfx_{k-1}} \leq \int_{\Zi^k}\fact{h}(\mrk{\bfx_{k-1}})^pd\mrk{\bfx_{k-1}}<\infty.
\end{equation*}

The first inequality in each row of (\ref{eq:contraction-norms}) is a  trivial consequence of the fact that $h$ and $g$ are non-negative. 
Consider first the case $r-l>0$. We have
\begin{align*}
\|h \star_{r}^l g \|^2_{L^2(  Z\times \Zi^{m-1})}&  = \int_{   Z\times \Zi^{m+2l-1}}h(\bfx_{k-r},\bfy_{r-l},
 \bfz_{l})g(\bfx'_{q-r},\bfy_{r-l}, \bfz_{l})\\
 &\hspace{3cm }  h(\bfx_{k-r},\bfy_{r-l},\bfz'_{l})g(\bfx'_{q-r},\bfy_{r-l},\bfz'_{l})d\mu^{m+2l}
\end{align*}
where we assume that the integration over $Z$ is performed on the variable $y_{1}$.
Call $t_{1}$ the spatial component of $y_{1}$. Since $ t_1 $ appears in the argument of each of the four kernels, we can introduce the change of variables such that $ t_1 $ is subtracted from every other coordinate, that is: $ t_1 \mapsto  t_1 , \bfx_{k-r}- t_1\mapsto \bfx_{k-r} ,\bfx'_{q-r}- t_1\mapsto\bfx'_{q-r} , \mrk\bfy_{r-l-1}- t_1\mapsto \mrk\bfy_{r-l-1} , \bfz_{l}- t_1\mapsto  \bfz_{l} ,\bfz'_{l}- t_1\mapsto \bfz'_{l} $. This yields

\begin{align*}
\|h \star_{r}^l g \|^2_{L^2(   Z\times \Zi^{m-1})}  &  = \leb(X) \int_{    \mrk {\Zi^{m+2l-1}}}\fact h(\bfx_{k-r},\mrk\bfy_{r-l-1},\bfz _{l})\fact g(\bfx'_{q-r},\mrk\bfy_{r-l-1}, \bfz_{l}) \\ & \hspace{4cm} \fact h(\bfx_{k-r},\mrk\bfy_{r-l-1},\bfz'_{l})\fact g(\bfx'_{q-r},\mrk\bfy_{r-l-1},\bfz'_{l})d\mrk \mu^{m+2l-1},\\
\end{align*}
where we have used the positivity and stationarity of $\fact h, \fact g$. Now we apply the Cauchy-Schwarz inequality to deduce that

%
%

\begin{eqnarray*}
&&\|h \star_{r}^l g \|^2_{L^2(   Z\times \Zi^{m-1})} \\ 
&&\leq \leb(X)\hspace{-1cm} \int\limits_{ \mrk\bfy_{r-l-1},\bfx_{k-r},\bfx_{q-r}' \in\mrk { \Zi^{m-1}} }\hspace{-1cm}d\mrk \mu^{m-1}\sqrt{\int_{ \bfz_{l},\bfz'_{l}\in  \Zi^{2l}}d\mu^{2l} \fact h^2(\bfx_{k-r},\mrk\bfy_{r-l-1},\bfz _{l})\fact h^2(\bfx_{k-r},\mrk\bfy_{r-l-1},\bfz'_{l}) }\\&& \hspace{4cm}\sqrt{\int_{ {\tilde \bfz}_{l}, {\tilde \bfz}'_{l}\in  \Zi^{2l}}d\mu^{2l} \fact g^2(\bfx'_{q-r},\mrk\bfy_{r-l-1}, {\tilde \bfz}_{l})\fact g^2(\bfx'_{q-r},\mrk\bfy_{r-l-1}, {\tilde \bfz}'_{l})}\\ 
%
&&=\leb(X) \int\limits_{\mrk\bfy_{r-l-1}\in \mrk{ \Zi^{r-l-1}}}d\mrk \mu^{r-l-1}\int_{\bfx_{k-r}\in  \Zi^{k-r}}d\mu^{k-r}\left(\int_{ \bfz_{l}\in  \Zi^l} d\mu^l \fact h^2(\bfx_{k-r},\mrk\bfy_{r-l-1},\bfz _{l})   \right)\\&&  \hspace{3cm} \int_{\bfx_{q-r}'\in  \Zi^{q-r}}d\mu^{q-r}\ \left(\int_{ \tilde\bfz_{l}\in \Zi^l} d\mu^l \fact g^2(\bfx'_{q-r},\mrk\bfy_{r-l-1},\tilde\bfz _{l}) \right)\\
&&\leq  \leb(X)\sqrt{\int_{\mrk\bfy_{r-l-1}\in\mrk{ \Zi^{r-l-1}}}d\mrk \mu^{r-l-1}\left( \int_{\bfx_{k-r}\in  \Zi^{k-r},\bfz_{l}\in  \Zi^l} d\mu^{k-r+l}\fact h^2(\bfx_{k-r},\mrk\bfy_{r-l-1}, \bfz_{l})\right)^2 }\\&&\hspace{1cm} \sqrt{ \int_{\mrk{\tilde\bfy}_{r-l-1}\in \mrk{ \Zi^{r-l-1}}}d\mrk\mu^{r-l-1} \left(\int_{\bfx'_{q-r}\in \Zi^{q-r},\bfz_{l}\in  \Zi^l}d\mu^{q-r+l} \fact g^2(\bfx'_{q-r},\mrk{\tilde \bfy}_{r-l-1}, \bfz_{l})  \right)^2}\\
&& =: \leb(X)\sqrt{I_{1}}\sqrt{I_{2}}
\end{eqnarray*}
where the definitions of $I_{1}$ and $I_{2}$ are obvious from the context. Now we call $\bft_{k-r+l}$ the $k-r$ spatial variables of $\bfx_{k-r}$ concatenated with the $l$ spatial variables of $\bfz_{l}$, we have 
\begin{align}
\nonumber I_{1} &= \int_{\mrk\bfy_{r-l-1}\in\mrk \Zi^{r-l-1}}\hspace{-2 cm}d\mrk \mu^{r-l-1}\hspace{.5cm}\left( \int_{\bfx_{k-r}\in  \Zi^{k-r},\bfz_{l}\in  \Zi^l}\hspace{-2cm} d\mu^{k-r+l}\kappa_{k-r+l}(\bft_{k-r+l})\kappa_{k-r+l}(\bft_{k-r+l})^{-1}\fact h^2(\bfx_{k-r},\mrk\bfy_{r-l-1}, \bfz_{l})\right)^2\\
\label{eq: CS1} &\leq \int_{\mrk\bfy_{r-l-1}\in\mrk \Zi^{r-l-1}}\hspace{-2 cm}d\mrk \mu^{r-l-1}\hspace{.5cm}\int_{\bfx_{k-r}\in  \Zi^{k-r},\bfz_{l}\in  \Zi^l}\hspace{-2cm}  d\mu^{k-r+l}\kappa_{k-r+l}(\bft_{k-r+l})\kappa_{k-r+l}(\bft_{k-r+l})^{-2}\fact h^4(\bfx_{k-r},\mrk\bfy_{r-l-1}, \bfz_{l})\\
\nonumber &\leq A_{4}(\fact h).
\end{align}
where we have used Cauchy-Schwarz inequality in the probability space $$(\Zi^{k-r+l},\kappa_{k-r+l}(\bft_{k-r+l})d\mu^{k-r+l}).$$ Performing similar computations for $I_{2}$, it follows that $\|h\star_{r}^l g\|^2_{L^2(   Z\times \Zi^{m-1})} \leq \leb(X)\sqrt{A_{4}(\fact h) A_{4}(\fact g)}$.

\smallskip

\noindent We now consider the more difficult case $r=l\geq 1, r<k$. For $\bfz_{r},\bfz'_{r}\in (\mathbb{R}^d)^r$, denote resp. by $t_{1}\in \mathbb{R}^d$ and $t_{1}'\in \mathbb{R}^d$ the spatial variables of $z_{1}$ and $z_{1}'$ (meaning that $z_{1}=(t_{1},m_{1}), z_{1}'=(t_{1}',m_{1}')$ for some $m_{1},m_{1}'\in M$).
\begin{align*}
\|h \star_{r}^{r} g \|^2_{L^2( Z\times \Zi^{m-1})}&  = \int_{Z\times \Zi^{k+q-2r-1}}d\mu^{k+q-2r}\left(\int_{Z^{r}}d\mu^r h(\bfx_{k-r},\bfz_{r})g(\bfx'_{q-r},\bfz_{r})\right)^2\\
&= \int_{Z\times \Zi ^{k+q-1}}d\mu^{k+q}h(\bfx_{k-r},\bfz_{r})g(\bfx'_{q-r},\bfz_{r})  h(\bfx_{k-r},\bfz'_{r})g(\bfx'_{q-r},\bfz'_{r})\\
&  = \int_{Z\times \Zi ^{k+q-1}}d\mu^{k+q}h(\bfx_{k-r}-t_{1},\bfz_{r}-t_{1})g(\bfx'_{q-r}-t'_{1}+t'_{1}-t_{1},\bfz_{r}-t_{1})  \\ 
&\hspace{3cm} h(\bfx_{k-r}-t_{1}+t_{1}-t_{1}',\bfz'_{r}-t_{1}')g(\bfx'_{q-r}-t_{1}',\bfz'_{r}-t_{1}')\\
\end{align*}
\noindent where $x_{1}=(t_{1},m_{1})$ is the variable restricted to $Z$, and we make the change of variables
\begin{align*}
\varphi\, \, :\, \, & t_{1}\mapsto t_{1}\\
&\bfx_{k-r}-t_{1}\mapsto\bfx_{k-r},\\
& \bfx_{q-r}'-t_{1}'\mapsto\bfx_{q-r}',\\
& t_{1}'-t_{1}\mapsto t_{1}',\\
& \mrk\bfz_{r-1}-t_{1}\mapsto\mrk\bfz_{r-1},\\
& \mrk\bfz'_{r-1}-t_{1}'\mapsto\mrk\bfz'_{r-1}.\\
\end{align*}
whose Jacobian is a triangular matrix with unit diagonal, and $\varphi$ satisfies $\varphi(Z\times Z^{k+q-1})=X\times \mrk{\Zi^{k+q-1}}$. One has that
\begin{align*}
\label{eq:ugly}
\|h \star_{r}^{r} g \|^2_{L^2( Z\times \Zi^{m-1})}&= \leb(X)\int_{\mrk{\Zi^{k+q-1}  } }d\mrk \mu^{k+q-1} \fact h(\bfx_ {k-r}, \mrk \bfz_{r-1})\fact g(\bfx_{q-r}'+t_{1}', \mrk \bfz_{r-1})\\
\nonumber &  \hspace{3cm}\fact h(\bfx_{k-r}-t_{1}', \mrk \bfz_{r-1}')\fact g(\bfx_{q-r}', \mrk \bfz_{r-1}').\\
\end{align*}
Applying Cauchy-Schwarz, and exploiting once again the positivity of the involved kernels, yields the estimate
%
%
%

\begin{align}
\|h \star_{r}^{r} g \|^2_{L^2( Z\times \Zi^{m-1})} &   \leq \leb(X)  \sqrt{\int_{  {\mrk \Zi } ^{k+q-1}} d\mrk \mu^{k+q-1} \fact h(\bfx_ {k-r}, \mrk \bfz_{r-1})\fact g^2(\bfx_{q-r}'+t_{1}', \mrk \bfz_{r-1})\fact g(\bfx_{q-r}', \mrk \bfz_{r-1}')}\\
\label{eq:bound}
 &\hspace{1.5cm}  \sqrt{\int_{ {\mrk  \Zi  } ^{k+q-1}} d\mrk \mu^{k+q-1} \fact h(\bfx_ {k-r}, \mrk \bfz_{r-1})\fact h^2(\bfx_{k-r}-t_{1}', \mrk \bfz_{r-1}')\fact g(\bfx_{q-r}', \mrk \bfz_{r-1}')}\\
 & =:\leb(X)\sqrt{I_{1}}\sqrt{I_{2}}, 
\end{align}
where the definition of $I_{1}$ and $I_{2}$ is clear from the context. 
Assume first that $q>r$.  Calling $s_{1}'$ the first spatial variable of $\bfx_{q-r}'$, the term under the first square root satisfies

\begin{align}
\nonumber I_{1}& \leq \hspace{-1cm} \int\limits_{s_{1}',t_{1}',\bfx_{k-r}, \mrk \bfz_{r-1}'\in (\mathbb{R}^d )^2\times{\mrk{\Zi}^{k-1}}}\hspace{-1cm}d\leb^2 d\mrk \mu^{k-1}\sqrt{\int\limits_{\mrk\bfx_{q-r-1}', \mrk \bfz_{r-1}\in\mrk{\mrk{\Zi^{q-2}}}}d\mrk{ \mu}^{q-r-1}d\mrk \mu^{r-1}\fact g^4(s_{1}'+t_{1}',\mrk\bfx_{q-r-1}'+t_{1}', \mrk \bfz_{r-1})}\\
\label{eq:s>r}& \hspace{3cm}\sqrt{\int\limits_{{\mrk{\tilde\bfx}}_{q-r-1}', \mrk{\tilde \bfz}_{r-1}\mrk{\mrk{\Zi^{q-2}}}}d\mrk{ \mu}^{q-r-1}d\mrk \mu^{r-1}\fact h^2(\bfx_{k-r}, \mrk{\tilde \bfz}_{r-1})\fact g^2(s_{1}',\mrk{\tilde \bfx}_{q-r-1}', \mrk \bfz_{r-1}')},
\end{align}
where the double underscore means that we are integrating on a set  where there are two more mark variables than spatial variables. Now, make the changes of variables $ \mrk\bfx_{q-r-1}'+t_{1}'\mapsto\mrk \bfx_{q-r-1}'$, and then $s_{1}'+t_{1}'\mapsto t_{1}'$ in their respective integrals . We have
\begin{align}
\label{eq:ugly}
I_{1}& \leq \int\limits_{s_{1}',t_{1}',\bfx_{k-r}, \mrk \bfz_{r-1}' \in (\mathbb{R}^d )^2\times\mrk{\Zi^{k-1}}}d\leb^2 d\mrk \mu^{k-1}\sqrt{\int_{\mrk\bfx_{q-r-1}', \mrk \bfz_{r-1}\in\mrk{\mrk{\Zi^{q-2}}}}d\mrk{ \mu}^{q-r-1}d\mrk \mu^{r-1}\fact g^4(t_{1}',\mrk\bfx_{q-r-1}', \mrk \bfz_{r-1})}\\
\nonumber &\hspace{3cm}\sqrt{\int_{\mrk{\tilde \bfx}_{q-r-1}', \mrk{\tilde \bfz}_{r-1}\in\mrk{\mrk{\Zi^{q-2}}}}d\mrk{ \mu}^{q-r-1}d\mrk \mu^{r-1}\fact h^2(\bfx_{k-r}, \mrk{\tilde \bfz}_{r-1})\fact g^2(s_{1}',\mrk{\tilde \bfx}_{q-r-1}', \mrk \bfz_{r-1}')}.
\end{align}
We can now separate the integrals: writing ${\bf t}_{k-r}$  and ${\bf t}_{r-1}$ for the set of spatial coordinates of ${\bf x}_{k-r}$ and $\mrk \bfz_{r-1}'$, respectively, we deduce that
\begin{align*}
I_{1} \leq & \int_{t_{1}'\in (\mathbb{R}^d )}d\leb \sqrt{\int_{\mrk\bfx_{q-r-1}', \mrk \bfz_{r-1}\in\mrk{\mrk{\Zi^{q-2}}}}d\mrk{ \mu}^{q-r-1}d\mrk\mu^{r-1}\fact g^4(t_{1}',\mrk\bfx_{q-r-1}', \mrk \bfz_{r-1})}\\
& \int_{\bfx_{k-r}\in\Zi ^{k-r}}d\mu^{k-r}\sqrt{\int_{ \mrk{\tilde \bfz}_{r-1} \in\mrk{\Zi^{r-1}}}d\mrk \mu^{r-1}\fact h^2(\bfx_{k-r}, \mrk{\tilde \bfz}_{r-1})} \\
& \int_{s_{1}', \mrk{ \bfz}_{r-1}'\in \mathbb{R}^d\times \mrk{\Zi^{r-1}}}d\leb d\mrk \mu^{r-1} \sqrt{\int_{\mrk{\tilde \bfx}_{q-r-1}'\in\mrk{\Zi^{q-r-1}}}\fact g^2(s_{1}',\mrk{\tilde \bfx}_{q-r-1}', \mrk{\bfz}_{r-1}')d\mrk \mu^{q-r-1}}\\
 & \leq   \sqrt{\int_{t_{1}'\in (\mathbb{R}^d )}d\leb \kappa_{1}(t_{1}')^{-1} \int_{\mrk\bfx_{q-r-1}', \mrk \bfz_{r-1}\in\mrk{\Zi^{q-2}}}d\mrk{ \mu}^{q-r-1}d\mrk\mu^{r-1}\fact g^4(t_{1}',\mrk\bfx_{q-r-1}', \mrk \bfz_{r-1})}\\
&\sqrt{  \int_{\bfx_{k-r}\in\Zi ^{k-r}} d\mu^{k-r}\kappa_{k-r}(\bft_{k-r})^{-1}{\int_{ \mrk{\tilde \bfz}_{r-1} \in\mrk{\Zi^{r-1}}}d\mrk \mu^{r-1}\fact h^2(\bfx_{k-r}, \mrk{\tilde \bfz}_{r-1})} }\\
& \sqrt{  \int_{(s_{1}', \mrk{\bfz}_{r-1}')\in  {\Zi^{r}}}d \mu^{r} \kappa_{r}(s_{1}',\bft_{r-1})^{-1}{\int_{\mrk{\tilde \bfx}_{q-r-1}' \in\mrk{\Zi^{q-r-1}}}\fact g^2(s_{1}',\mrk{\tilde \bfx}_{q-r-1}',\mrk{ \bfz}_{r-1}')}d\mrk \mu^{q-r-1}}\\
& \leq \sqrt{A_{4}(\fact g)A_{2}(\fact h)A_{2}(\fact g)}
\end{align*}
where we have used the Cauchy-Schwarz in the probability spaces $$(\mathbb{R}^d, \kappa(t)dt), \quad (\Zi^{k-r}, \kappa_{k-r}(\bft_{k-r})d\nu^{k-r}),\quad( {\Zi^{r}},  d \mu^{r} \kappa_{r}(\bft_{r})),$$ exactly as we did in (\ref{eq: CS1}). If $r=l=q<k$ , there is no such variable as $s_{1}'$ in the integral $I_{1}$, thus the previous method does not work, still there is an easier procedure.
From (\ref{eq:s>r}) we directly get, with obvious shorthand notation,
\begin{equation*}
I_{1} \leq \int_{t_{1}',\bfx_{k-r}, \bfz'_{r-1}}\sqrt{\int_{\bfz_{r-1}} \fact g^4(\bfz_{r-1}) }\sqrt{\int_{ \tilde \bfz_{r-1} }\fact h^2(\bfx_{k-r},\tilde \bfz_{r-1})\fact g^2(\bfz'_{r-1}) } \leq A_{4}(\fact g) A_{2}(\fact h) A_{2}(\fact g)
\end{equation*} 
Remark that the expression of $I_{2}$ can be obtained from that of $I_{1}$ by inverting ``$h$'' and ``$g$'', ``$+t_{1}'$'' and ``$-t_{1}'$'', ``$\bfx'$'' and ``$\bfx$'', ``$\bfz$'' and ``$\bfz'$'',  and ``$k$'' and ``$q$'', the fact that $q \leq k$ does not play any role. Thus we get the desired estimates. 
\qed 

\medskip

\begin{defi}{\rm In the light of the previous statement, we call \emph{rapidly decreasing function} any stationary function $h$ on $\Zi^k$ such that $A_{\kappa,p}(|\fact h|)<\infty$ for $p=2,4$ and some probability density $0 < \kappa \leq 1$. This terminology refers to the shape of $h$ when $d=2$: constant along diagonal lines and uniformly decaying sufficiently fast far from the diagonal.  Remark that since $\kappa^{-1} \geq 1$,  given a rapidly decreasing function $h$, $\fact h$ has finite $L^2$ and $L^4$ norm on $\mathbb{R}^d$.
}
\end{defi}

Lemma \ref{lm:inv-kernel} is an expression of the fact that, when one computes the norm of a  \emph{rapidly decreasing function} $h(x_{1},\dots,x_{k})$ on a product space $Z_{1}\times \dots \times Z_{k}$, or of  contractions of such functions, the order of magnitude of the result depends only on the measure of the smallest  $Z_{i}$.

\medskip

\begin{rem}{\rm A sufficient condition for a function $h$ to be rapidly decreasing is that $$| \fact h(\bft_{k-1}, {\bf m}_k) |\leq H(\bft_{k-1}),$$ where $H$ is a bounded function not depending on ${\bf m}_k$ and such that $H$ converges to zero at a subexponential speed as $\| {\bft}_{k-1} \|_{(\R^d)^{k-1}} \to \infty$.

}
\end{rem}

\section{Asymptotic normality for finite \WI expansions}\label{s:finite}

\medskip

{\bf N.B.} Throughout this section, we apply the same conventions outlined in Remark \ref{r:conventions}-(a).

\medskip

The analytical results of the previous section will be now used to deduce bounds on the speed of convergence of a variable with a finite \WI expansion -- under some specific scaling assumptions on the kernels. We will later apply these findings to \Uss. 
\begin{thm}
\label{th:general-clt} 
Let $\alpha_{\lambda}>0$ and $F_{\lambda}$ be  of the form (\ref{e:genseq}), where the multiple integrals are with respect to a Poisson measure on $Z$ with control $\mu_\lambda =\lambda \mu$, and the kernels $f_{i,\lambda}\in L^2_s(Z^{q_i}, \mu_\lambda^{q_i})$ are such that, for each $1 \leq i \leq k$, there is $\gamma_{i,\lambda}>0$ and $h_{i,\lambda}\in L^2(Z_\lambda^{q_i}, \mu^{q_i})$ verifying
\begin{equation}
\label{eq:scaling-form}
f_{i,\lambda}(\bfx_{i})=\gamma_{i,\lambda}h_{i,\lambda}(\alpha_{\lambda}\bfx_{i}).
\end{equation}
Assume that for each $1 \leq i \leq k$ there exists a nonnegative measurable function $h_{i}$ on $ \Zi^{q_i}$ such that
\begin{enumerate}
\item $h_{i}$ is a rapidly decreasing function not identically equal to $0$,
\item $| h_{i,\lambda}| \leq h_{i}$ on $Z_\lambda^{q_i}$, for every $\lambda>0$,
\item $\|h_{i,\lambda}\|_{L^p(Z^{q_i}_{\lambda};\mu^{q_{i}})}\sim  \|h_{i}\|_{L^p(Z^{q_i}_{\lambda};\mu^{q_{i}})}$ as $\lambda\to \infty$ for $p=2,4$.
\end{enumerate}
Define 
\begin{equation*}
m_{\lambda}=\lambda \alpha_{\lambda}^{-d}.
\end{equation*}
 Then
\begin{equation*}
\var(F_{\lambda})=\sigma_{\lambda}^2\sim   \alpha_{\lambda}^d \sum_{i=1}^{k}q_{i}! \gamma_{i,\lambda}^2 m_{\lambda}^{q_i} \|\fact{h}_{i}\|^2_{L^2(\mrk{Z_\lambda^{q_{i}-1}}, \mrk{\mu^{q_i-1}})}\asymp   \alpha_{\lambda}^d \max_{i=1}^{k}  \gamma_{i,\lambda}^2 m_{\lambda}^{q_i}.
\end{equation*}
\begin{equation*}
B_{3}(\tilde F_{\lambda};1) \leq C \sqrt{\omega_{\lambda}+\omega_{\lambda}'}
\end{equation*}
with 
\begin{align}
\label{eq:omega}
\omega_{\lambda}:&=\frac{1}{\sigma_{\lambda}^4}\alpha_{\lambda}^{d}\max_{(*)}\{(\gamma_{i,\lambda}\gamma_{j,\lambda})^2 m_\lambda^{q_{i}+q_{j}-r+l}\} \\
\omega'_\lambda:&=\frac{1}{\sigma_{\lambda}^4}\alpha_{\lambda}^d \max_{i=1,\dots,k}\{\gamma_{i,\lambda}^4 m_\lambda^{q_{i}}\},
\end{align}
where $\displaystyle{\max_{(*)}}$ is defined as in (\ref{e:B3}). One also has the estimate in the Wasserstein distance:
\begin{align*}
d_{W}(\tilde F,N) \leq C\sqrt{\omega_{\lambda}+\omega'_{\lambda}}.\\
\end{align*}
\end{thm}
\begin{proof}
 Using in sequence (\ref{eq:Lp-scaling}) and (\ref{eq:contraction-norms2}), one deduces the following estimates on the variance 
\begin{align*}
\sigma_{\lambda}^2&=\sum_{i=1}^{k}q_i ! \|f_{i,\lambda}\|^2_{L^2(Z^{q_i};\mu_{\lambda}^{q_i})}=\sum_{i=1}^{k} q_i !\gamma_{i,\lambda}^2m_{\lambda}^{q_i}\|h_{i,\lambda}\|^2_{L^2(Z_{\lambda}^{q_i};\mu^{q_i})}\\
 &\sim  \alpha_{\lambda}^d \sum_{i=1}^{k}q_{i}! \gamma_{i,\lambda}^2 m_{\lambda}^{q_i} \|\fact{h}_{i}\|^2_{L^2(\mrk{Z_\lambda^{q_{i}-1}}, \mrk{\mu^{q_i-1}})} .\end{align*}
Now, in $B_{3}(\tilde F_{\lambda}; 1)$, every kernel is normalized by $\sigma_{\lambda}$, whence every squared $L^2$ contraction norm must be divided by $\sigma_{\lambda}^4$, as well as the $4$th power of the $L^4$-norm: 
\begin{equation*}
B_{3}(\tilde F_{\lambda}; 1)^2 \leq  \frac{1}{\sigma_{\lambda}^4}\max_{(*)}\|f_{i,\lambda}\star_{r}^l f_{j,\lambda}\|_{L^2(Z^{q_i+q_j-r-l};\mu_{\lambda}^{q_i+q_j-r-l})}^2+\frac{1}{\sigma_{\lambda}^4}\max_{i=1,...,k}\|f_{i,\lambda}\|^4_{L^4(Z^{q_i};\mu_{\lambda}^{q_i})}.
\end{equation*}
Using (\ref{eq:contraction-scaling}) and (\ref{eq:contraction-norms}), for each $i,j,r,l$ appearing in the argument of $\displaystyle{\max_{(*)}}$ we have 
\begin{eqnarray*}
&& \|f_{i,\lambda}\star_{r}^l f_{j,\lambda}\|^2_{L^2(Z^{q_i+q_j-r-l}; \mu_\lambda^{q_i+q_j-r-l})} \\&& = (\gamma_{i,\lambda}\gamma_{j,\lambda})^2 m_{\lambda}^{q_{i}+q_{j}-r+l} \|h_{i,\lambda}\star_{r}^l h_{j,\lambda}\|^2_{L^2(\alpha_\lambda Z^{q_i+q_j-r-l}; \mu^{q_i+q_j-r-l})} \leq C  (\gamma_{i,\lambda}\gamma_{j,\lambda})^2 m_{\lambda}^{q_{i}+q_{j}-r+l} \alpha_{\lambda}^{d}
\end{eqnarray*}
because $|h_{i,\lambda}|\leq h_{i}$ and $h_{i}$ is a rapidly decreasing function. Turning to the $L^4$ norm, using again (\ref{eq:Lp-scaling}) and (\ref{eq:contraction-norms2}) we have, for $1 \leq i \leq k$,
\begin{equation*}
\|f_{i}\|^4_{L^4(Z^{q_i};\mu_{\lambda}^{q_i})} \leq \gamma_{i,\lambda}^4 m_{\lambda}^{q_{i}}\|h_{i,\lambda}\|^4_{L^4(\alpha_\lambda Z^{q_i};\mu^{q_i})} \leq C' \gamma_{i,\lambda}^4 m_{\lambda}^{q_{i}}\alpha_{\lambda}^d 
\end{equation*}
because $\fact {h_{i}}\in L^4(\mrk{\Zi^{q_{i}-1}})$ (as already observed, this is an easy consequence of the assumption $A_{4}(\fact h_{i})<\infty$).
All the bounds are easily deduced, and the estimate on the Wasserstein distance follows from Theorem \ref{t:main2}.

\end{proof}

One can use this result to directly have asymptotic normality for random variables having a finite \WI expansions. For variables with infinite expansion, one can for instance use a truncation argument -- see e.g. \cite{Schulte}. In the sequel, the previous findings are applied in order to deduce asymptotic normality for \Uss.

\section{$U$-statistics with a stationary rescaled kernel}\label{s:stat}

{\bf N.B.} In this section, the notation and framework of Problem \ref{prob:Ustat} prevail. 

\medskip

We assume that $F_{\lambda}$ is a square-integrable \Us $\,$  of the form (\ref{eq:scaled-Ustat}) with $h$ a rapidly decreasing function on $\Zi^k$ (recall that a rapidly decreasing function is stationary by definition). Recall that $\mu = \ell \otimes \nu$, where $\nu$ is a probability measure on the marks space $M$. For $1\leq i \leq k$, we set
\begin{equation}
\label{eq:def-projection}
|h|_{i}(\bfx_{i})=\binom{k}{i}\int_{\Zi^{k-i}}|h|(\bfx_{i},\bfx_{k-i})d\mu^{k-i}, \quad \bfx_{i}\in \Zi^{i}.
\end{equation}

Note that the kernels $|h|_{i}$ are also stationary (just use the invariance of Lebesgue measure). We say that $|h|$ has \emph{rapidly decreasing projections} if the quantity $A_{\kappa, p}(|h_{i}| ) = A_{ p}(|h_{i}| )$ (as defined in (\ref{e:akappa})) is finite for some density $\kappa$, and every $1 \leq i \leq k$ and $ p=2,4$. The following statement is the main achievement of the present section.

\begin{rem}{It is easily seen that, if $h$ has rapidly decreasing projections, then the kernels $h_i$, $i=1,...,k$, defined in (\ref{e:lastpenrose}) necessarily satisfy Point (1) and Point (2) in Assumption \ref{a:tech}. In order to apply our results, in what follows we shall also require that the kernels $|h|_i$ (and therefore the kernels $h_i$) verify Assumption \ref{a:tech}-(3): for instance, a sufficient condition for this assumption to hold is that $| h |\leq H$, where the function $H$ only depends on the spatial variables ${\bf t}_k$, is bounded and has compact support.
}
\end{rem}

\begin{thm}\label{th:CLT-Ustat}
In the framework of this section, assume that $|h|$ has rapidly decreasing projections, and that the kernels $|h|_{i}(\bfx_{i})$, $i=1,...,k$ verify Part (3) of Assumption \ref{a:tech}. Define 
\begin{equation*}
m_{\lambda}=\lambda\alpha_{\lambda}^{-d}.
\end{equation*}Then 
\begin{equation*}
\var(F_{\lambda})\asymp  \alpha_{\lambda}^dm_{\lambda}^{2k-1}\max(1,m_{\lambda}^{-k+1}),
\end{equation*}
and
\begin{equation*}
d_{W}(\tilde F_\lambda,N) \leq C' \alpha_{\lambda}^{-d/2}\sqrt{\max(1,m_{\lambda}^{-k},m_{\lambda}^{-2(k-1)})}
\end{equation*}
for some $ C'>0$.
\end{thm}

\begin{proof}

According to (\ref{e:lastpenrose}) the kernels $f_{i,\lambda}$ of $F_{\lambda}$'s \WI decomposition are given by: for $1 \leq i \leq k$,

\begin{equation*}
f_{i,\lambda}(\bfx_{i})=\binom{k}{i}\int_{Z^{k-i}}h(\alpha_{\lambda}( \bfx_{i},\bfx_{k-i}))d \mu_{\lambda}^{k-i}=\binom{k}{i}(\lambda \alpha_{\lambda}^{-d})^{k-i}\int_{\alpha_{\lambda}Z^{k-i}}h(\alpha_{\lambda}\bfx_{i},\bfx_{k-i})d\mu^{k-i},
\end{equation*}
whence $f_{i,\lambda}$
 is of the form (\ref{eq:scaling-form}), with 
\begin{align*}
\gamma_{i,\lambda}&=m_{\lambda}^{k-i},\\
h_{i,\lambda}(\bfx_{i})&=\binom{k}{i}\int_{\alpha_{\lambda}Z^{k-i}}h(\bfx_{i},\bfx_{k-i})d\mu^{k-i}.
\end{align*}
Our purpose is now is to verify Point 1 and Point 2 in Theorem \ref{th:general-clt} with $h_{i} = |h|_i$, as defined in (\ref{eq:def-projection}). Point 1 is an immediate consequence of the estimates
\begin{align*}
0 \leq | h_{i,\lambda}(\bfx_{i})| \leq \binom{k}{i} \int_{\Zi^{k-i}}|h|(\bfx_{i},\bfx_{k-i})d\mu^{k-i}=|h|_{i}(\bfx_{i}).
\end{align*}
To prove Point 2
\begin{align*}
\|h_{i,\lambda}\|_{L^p(\alpha_{\lambda}Z^i;\mu^i)}^p&=\int_{Z_\lambda^i}|h|_{i,\lambda}^p(\bfz_{i})d\mu^i=\int_{Z_\lambda^i}\int_{Z_\lambda^{p(k-i)}}\prod_{q=1}^p |h|(\bfz_{i},\bfz^{(q)}_{k-i})d\mu^{i+p(k-i)}\\
&=\int_{Z_\lambda} d\mu \int_{Z_\lambda^{i-1+p(k-i)}}\prod_{q=1}^p  |\fact h|(\bfz_{i-1}-z_{1},\bfz_{k-i}^{(q)}-z_{1})d\mu^{i-1+p(k-i)},
\end{align*}
where $\bfz_{i}$ is decomposed in $(z_{1},\bfz_{i-1})$ and the variable $z_{1}$ has been subtracted to every other variable present in the integral, exploiting the stationarity of $h$. Writing $X_\lambda = \alpha_\lambda X$, we now make the change of variables
\begin{equation*}
\varphi(z_{1},\bfz_{i-1},\bfz_{k-i}^{(1)},\dots,\bfz_{k-i}^{( p)})=(z_{1},\bfz_{i-1}-z_{1},\bfz_{k-i}^{(1)}-z_{1},\dots,\bfz_{k-i}^{( p )}-z_{1}),
\end{equation*}
verifying $$\check X_{\lambda}\times \mrk{\check Z_\lambda^{i-1+p(k-i)} }\subseteq \varphi(Z_\lambda^{i-1+p(k-i)}) \subseteq X_{\lambda}\times \mrk{  \hat Z_\lambda^{i-1+p(k-i)}},$$
 we have 
\begin{eqnarray*}
\leb(\check X_\lambda)\int_{\check Z_\lambda^{i-1+p(k-i)}} \prod_{q=1}^p|\fact h|(\bfz_{i-1},\bfz^{(q)}_{k-i})d\mu^{i+q(k-i)}& \leq& \|h_{i,\lambda}\|_{L^p(\alpha_{\lambda}Z^i;\mu^i)}^p \\ &\leq& \leb(X_\lambda)\int_{\hat Z_\lambda^{i-1+p(k-i)}} |\fact h|(\bfz_{i-1},\bfz^{(q)}_{k-i})d\mu^{i+q(k-i)}.
\end{eqnarray*}
Since $0<\leb(X)<\infty$, one has that   $\leb(\check X_{\lambda})\asymp \mu(Z_{\lambda}) \asymp \alpha_{\lambda}^d$ and for $\lambda$ sufficiently large
\begin{align*}
0<\int_{\check Z_{\lambda}^{i-1+p(k-i)}}\prod_{q=1}^p |\fact h|(\bfz_{i-1},\bfz^{(q)}_{k-i})d\mu^{i+q(k-i)}& \leq \frac{\|h_{i,\lambda}\|_{L^p(\alpha_{\lambda}Z^i;\mu^i)}^p}{\alpha_{\lambda}^d}\\
& \leq \int_{\hat Z_{\lambda}^{i-1+p(k-i)}} \prod_{q=1}^p|\fact h|(\bfz_{i-1},\bfz^{(q)}_{k-i})d\mu^{i+q(k-i)}.
\end{align*}
Both sides converge to  $\|\fact h_{i}\|_{L^p(\Zi^i;\mu^i)}^p\leq A_{p}(\fact h_{i})<\infty$ by hypothesis. We have with similar computations

\begin{align*}
0<\int_{\check Z_{\lambda}^{i-1}\times \Zi^{p(k-i)}}
\prod_{q=1}^p \fact h(\bfz_{i-1},\bfz^{(q)}_{k-i})d\mu^{i+q(k-i)}
&
 \leq \frac{\|h_{i}\|_{L^p(Z_{\lambda}^i;\mu^i)}^p}{\alpha_{\lambda}^d}\\& \leq \int_{\hat Z_{\lambda}^{i-1}\times \Zi^{p(k-i)}} \prod_{q=1}^p \fact h(\bfz_{i-1},\bfz^{(q)}_{k-i})d\mu^{i+q(k-i)}
\end{align*}
and both sides converge to   $\|\fact h_{i}\|_{L^p(\Zi^i:\mu^i)}^p$, whence we have indeed $\|h_{i,\lambda}\|^p_{L^p}\sim \|h_{i}\|^p_{L^p}\sim \alpha_{\lambda}^d \|\fact h_{i}\|_{L^p}^p$, and the Point 2 is verified. Using the notation of Th. \ref{th:general-clt}, we have that (\ref{eq:omega}) is verified, and for some constant $C>0$
\begin{equation*}
\var(\tilde F_{\lambda})\asymp \alpha_{\lambda}^d\times \max_{i=1,\dots,k}(m_{\lambda}^{{i}}m_{\lambda}^{2(k-i)})\asymp  \alpha_{\lambda}^dm_{\lambda}^{2k-1}\max_{i=1,\dots,k}(m_{\lambda}^{-(i-1)})\asymp \alpha_{\lambda}^dm_{\lambda}^{2k-1}\max(1,m_{\lambda}^{-k+1}),
\end{equation*}
and
\begin{equation*}
d_{W}(\tilde F_\lambda,N) \leq C\sqrt{\omega_{\lambda}+\omega_{\lambda}'}
\end{equation*}
where, writing $\sigma_\lambda^2 = \var(\tilde F_{\lambda})$,
\begin{align*}
\omega_{\lambda}:&=\frac{1}{\sigma_{\lambda}^4}\alpha_{\lambda}^{d}\max_{1}\gamma_{i,\lambda}^2 \gamma_{j,\lambda}^2 m_{\lambda}^{i+j-r+l}\\
&=\frac{1}{\sigma_{\lambda}^4}\alpha_{\lambda}^{d}\max_{1}m_{\lambda}^{2(2k-i-j)} m_{\lambda}^{{i}+{j}-r+l}\\
& =\alpha_{\lambda}^{-d} m_{\lambda}^{-2(2k-1)}m_{\lambda}^{4k}\frac{\max_{1}(m_{\lambda}^{-i-j-r+l})}{\max_{}(1, m_{\lambda}^{-k+1})^2}\\
& \asymp\alpha_{\lambda}^{-d} m_{\lambda}^{2}\frac{\max(m_{\lambda}^{-3k+1},m_{\lambda}^{-2})}{\max_{}(1, m_{\lambda}^{-k+1})^2}\\
&=\alpha_{\lambda}^{-d}\times\begin{cases}
1$ if $m_{\lambda}\geq 1\\
m_{\lambda}^{-2(k-1)}$ if $m_{\lambda}<1
\end{cases}
\end{align*}

and

\begin{align*}
\omega'_\lambda:&=\frac{1}{\sigma_{\lambda}^4}\max_{i=1,\dots,k}\gamma_{i,\lambda}^4 m_{\lambda}^{{i}}\alpha_{\lambda}^d \\
&=\alpha_{\lambda}^{-2d}m_{\lambda}^{-2(2k-1)}\max(1,m_{\lambda}^{1-k})^{-2}\max_{i=1,\dots,k}m_{\lambda}^{4(k-i)}m_{\lambda}^i \alpha_{\lambda}^d\\
&=\alpha_{\lambda}^{-d} m_{\lambda}^{-2(2k-1)}m_{\lambda}^{4k}\max(1,m_{\lambda}^{1-k})^{-2}m_{\lambda}^{-3}\max_{i} m_{\lambda}^{-3(i-1)}\\
&=\alpha_{\lambda}^{-d} m_{\lambda}^{-1} \max(1,m_{\lambda}^{1-k})^{-2}\max(1,m_{\lambda}^{-3(k-1)})\\
&=\alpha_{\lambda}^{-d}
\begin{cases}
m_{\lambda}^{-1}$ if $m_{\lambda}>1\\
m_{\lambda}^{-k}$ if $m_{\lambda}\leq 1.
\end{cases}
\end{align*}

Taking the overall maximum yields 
\begin{equation*}
d_{W}(\tilde F,N) \leq C \alpha_{\lambda}^{-d/2}\sqrt{\max(1,m_{\lambda}^{-k},m_{\lambda}^{-2(k-1)})}.
\end{equation*}

\end{proof}

The condition that $h$ has rapidly decreasing projections seems a bit abstract, so we give sufficient conditions below.
 The symbol $\kappa$ denotes a fixed  probability density taking values in $(0,1]$.

\begin{lemma}
\label{lm:integrability-condition}

Let $\fact h$ be a non-negative function on $\mrk{\Zi^{k-1}}, k\geq 1$.
Define for $p=2,4$
\begin{equation*}
A_{p}'(\fact h)=\int_{\mrk{\Zi^{k-1}}}\fact h^p(\mrk\bfx_{k-1})\kappa_{k-1}(\bft_{k-1})^{1-p}d\mrk\mu^{k-1}.
\end{equation*} Every projection 
\begin{equation*}
\fact h_{j}(\mrk\bfx_{j-1})=\binom{k}{j}\int_{\Zi^{k-j}}\fact h(\mrk\bfx_{j-1},\bfx_{k-j})d \mu^{k-j}, 1 \leq j \leq k,  \mrk\bfx_{j-1}\in \mrk\Zi^j, 
\end{equation*}
for $1 \leq j \leq k$
satisfies
\begin{equation}
\label{eq:B2leqB2prime}
A_{p}(\fact h_{j})\leq A_{p}'(\fact h), p=2,4.
\end{equation}

\end{lemma}

\begin{proof}

Holder's inequality yields, for $p=2,4, \mrk\bfx_{j-1}\in\mrk\Zi^{j-1}$,

\begin{align*}
\fact h_{j}^p(\mrk\bfx_{j-1})&=\left(\int_{\Zi^{k-j}}\kappa_{k-j}(\bft_{k-j})\kappa_{k-j}(\bft_{k-j})^{-1}\fact h(\mrk\bfx_{j-1},\bfx_{k-j})d\mu^{k-j}\right)^p \\
&\leq \int_{\Zi^{k-j}}\fact h^p(\mrk\bfx_{j-1},\bfx_{k-j}) \kappa_{k-j}(\bft_{k-j})^{-p}\kappa_{k-j}(\bft_{k-j}) d\mu^{k-j}\\
&=\int_{\Zi^{k-j}}\fact h^p (\bfx_{k-j},\mrk\bfx_{j-1})\kappa_{k-j}(\bft_{k-j})^{1-p}d\mu^{k-j},
\end{align*}
whence
\begin{align*}
 A_{p}(\fact h_{j})&=\int_{\mrk \Zi^{j-1}}\kappa_{j-1}(\bft_{j-1})^{-1}\fact h_{j}^p (\mrk\bfx_{j-1})d\mrk\mu^{j-1}\\
&\leq  \int_{\mrk\Zi^{k-1}}\kappa_{j-1}(\bft_{j-1})^{-1}\kappa_{k-j}(\bft_{k-j})^{1-p}\fact h^p (\mrk\bfx_{k-1})d\mrk\mu^{k-1}\\
& \leq \int_{\mrk\Zi^{k-1}}\kappa_{k-1}(\bft_{k-1})^{1-p}\fact h^p(\mrk\bfx_{k-1})d\mrk\mu^{k-1}=A_{p}'(\fact h).
\end{align*}

\end{proof}

As a consequence, $h$ has rapidly decreasing projections if $A_{p}'(\fact h)<\infty$ for $p=2,4$ for some $\kappa$ (this condition clearly implies that $h$ is itself a rapidly decreasing function).

\section{Asymptotic characterization of geometric \Uss}
\label{sec:geom-Ustat}

In this section we work in the framework of Problem \ref{prob:Ustat}, in the special case where $\alpha_{\lambda}= 1$ for every $\lambda$. As before, $h\in L^{1,2}(Z^k)$ is such that each $F_{\lambda}$ defined by (\ref{eq:scaled-Ustat}) is a square-integrable $U$-statistic. In the terminology of \cite{lesmathias} each 
$F_\lambda$ is a \emph{geometric \Us}. In what follows, we shall use the theory of {\it Hoeffding decompositions} (see e.g. Vitale \cite{Vitale}) in order to provide a complete characterization of the asymptotic behavior of the family $F_{\lambda}$, as $\lambda\to \infty$, yielding a substantial generalization of the results proved in \cite[Section 5]{lesmathias}. Observe that, according to (\ref{e:lastpenrose}),

\begin{equation}\label{e:cd}
F_\lambda = E[F_\lambda] + \sum_{i=1}^k F_{i,\lambda}:= E[F_\lambda] + \sum_{i=1}^k I_i(f_{i,\lambda}),
\end{equation}
where
\begin{equation}\label{e:filambda}
f_{i,\lambda}(\bfx_{i})=\lambda^{k-i}\binom{k}{i}\int_{Z^{k-i}}h(\bfx_{i},\bfx_{k-i})d\mu^{k-i}=:\lambda^{k-i}h_{i}(\bfx_{i}), \,\bfx_{i}\in Z^i,
\end{equation} 
and each multiple integral is realized with respect to the compensated Poisson measure $\hat{\eta}_\lambda = \eta_\lambda - \lambda \mu$.

 Also, since $X$ is a compact set of $\R^d$ and $\nu$ is a probability measure, one has that, for every $\lambda$, $\mu_\lambda(Z) = \lambda\ell(X)\nu(M)<\infty$ and $\eta_\lambda(Z)$ is a Poisson random variable of parameter $\mu_\lambda(Z)$. In order to state our main findings, we need the definition of a Gaussian measure. The reader is referred e.g. to \cite[Chapter 5]{PeTa} for an introduction to Gaussian measures and associated multiple integrals.

\begin{defi}{\rm

\begin{enumerate}

\item[\rm (i)] A {\it Gaussian measure} over $Z$ with control $\mu$ is a centered Gaussian family of the type $G  = \{G(B) : \mu(B) <\infty\}$, such that, for every $B,C$ verifying $\mu(B),\,\mu(C)<\infty$,
\[
E[G(B)G(C)] =\mu(B\cap C). 
\]

\item[\rm (ii)] Define $\hat{\mu}$ as the probability measure $\frac{\mu}{\mu(Z)}$. A {\it Gaussian measure} over $Z$ with control $\hat{\mu}$ is a centered Gaussian family of the type $\hat G  = \{\hat G(B) : \mu(B) <\infty\}$, such that, for every $B,C$ verifying $\mu(B),\,\mu(C)<\infty$,
\[
E[\hat G(B)\hat G(C)] =\hat{ \mu}(B\cap C). 
\]
\end{enumerate}
}
\end{defi}

\begin{rem}{\rm Using the self-similarity properties of the Gaussian distribution, one sees immediately that $\hat G \stackrel{\rm Law}{=} \mu(Z)^{-1/2} \times G$. In what follows, we shall maintain the distinction between the two Gaussian measures $G$ and $\hat G$ in order to facilitate the connexion with reference \cite{DyMa}.

}
\end{rem}

The following statement is the main result of this section. Note that Point 1 corresponds to Theorem 5.2 in \cite{lesmathias}.

\begin{thm}{\bf (Asymptotic characterization of geometric $U$-statistics)}
\label{th:geom-ustat}
Let $q_{1}\geq 1$ be the smallest integer such that $\| h_{q_{1}}\|_{L^2(\mu^{q_1})} >  0$, where the kernels $\{h_i : i=1,...,k\}$ are defined according to (\ref{e:filambda}). Then, as $\lambda\to\infty$,
\begin{equation*}
\var(F_{\lambda})\sim \var(F_{q_{1},\lambda})\sim c \lambda^{2k-q_{1}},
\end{equation*}
for some $c>0$,
and moreover:

\begin{enumerate}
\item If $q_{1}=1$, $\tilde F_{\lambda}$ converges in law to $N\sim \mathscr{N}(0,1)$ with an upper bound of the order $\lambda^{-1/2}$ on the Wasserstein distance.
\item If $q_{1}\geq 2$, then $\lambda^{q_1/2 -k} F_{\lambda}$ converges in distribution to 
\begin{equation}\label{e:lim}
V(q_1) :=  \mu(Z)^{q_1/2} \times  I^{\hat G}_{q_1}(h_{q_1})\stackrel{\rm Law}{=} I^{G}_{q_1}(h_{q_1}),
\end{equation}
where $I^{\hat G}_{q_1}$ (resp. $I^{ G}_{q_1}$) indicates a multiple Wiener-It\^o integral of order $q_1$, with respect to a Gaussian measure $\hat G$ (resp. $G$) on $Z$ with control $\hat{\mu} = \mu/\mu(Z)$ (resp. $\mu$).
\end{enumerate}
\end{thm}

\begin{rem}{\rm

\begin{enumerate}

\item It is a well-known fact that a non-zero multiple Wiener-It\^o integral (with respect to a Gaussian measure) of order strictly greater than one cannot have a Gaussian distribution (see e.g. Janson \cite[Chapter V]{Janson}).

\item Observe that 
\[
E[V(q_1)^2] = q_1! \binom{k}{q_1}^2 \int_{Z^{q_1}} \left( \int_{Z^{k-q_1}} h(\bfx_{q_1} , \bfx_{k-q_1}) \mu^{k-q_1}(d\bfx_{k-q_1})\right)^2\mu^{q_1}(d\bfx_{q_1}),
\]
which is consistent with the estimates on the variance contained in \cite[Section 5]{lesmathias}.

\item The content of Theorem \ref{th:geom-ustat} allows one to give a complete explanation of a counterexample provided in \cite[end of Section 5.1]{lesmathias}. Consider indeed a Poisson process $\eta_\lambda$ on $Z = X = [-1,1]$, with control measure given by $\lambda dx$. Define the kernel $f(x_1,x_2)$ on $Z^2$ as follows: $f(x_1,x_2)=1$, if $x_1x_2\geq 0$ and $f(x_1,x_2) = -1$, if $x_1x_2<0$. Then, the random variable $F_\lambda = \sum_{(x_1,x_2) \in \eta_{\lambda, \neq}} f(x_1,x_2)$ is a $U$-statistic of order 2 such that $q_1=2$, and Part 2 of Theorem \ref{th:geom-ustat} implies that $\lambda^{-1} F_\lambda$ converges in distribution to $V(2) = I_2^G(f)$, where $G$ is a Gaussian measure on $Z$ with control equal to the Lebesgue measure. Standard results on multiple stochastic integrals imply that $V(2)$ is a non-Gaussian random variable having the same law as
\[
\xi_1^2- 1 +\xi_2^2- 1 -2\xi_1\xi_2 \stackrel{\rm Law}{=} 2(\xi_1^2-1),
\]
where $(\xi_1,\xi_2)$ is a two-dimensional vector of i.i.d. centered Gaussian random variables with unit variance. In particular, one has that $E[V(2)^2] = 8$ and $E[V(2)^3] = 64$. 
\end{enumerate}
}
\end{rem}

\begin{proof}[Proof of Theorem \ref{th:geom-ustat}]
For each $1 \leq j \leq k$
\begin{equation*}
\var(F_{j,\lambda})=j!\|f_{j}\|^2_{L^2(Z^j; \mu_{\lambda}^j)}=c_{j} {\lambda}^{2k-j},
\end{equation*}
for some constant $c_{j}\geq 0$ (independent of $\lambda$) such that $c_{q_{1}}>0$.
By orthogonality, 
\begin{equation*}
\var(F_\lambda)=\lambda^{2k}\sum_{j=1}^k c_{j}\lambda^{-j}=c_{q_{1}}\lambda^{2k-q_{1}}(1+\sum_{q_{1}+1}^k c_{j}c_{q_{1}}^{-1}\lambda^{q_{1}-j})=c_{q_{1}}\lambda^{2k-q_{1}}(1+O(\lambda^{-1})).
\end{equation*}
It follows that, as $\lambda\to\infty$ and writing $\tilde{F}_{q_1,\lambda} = F_{q_1,\lambda}/\sqrt{\var(F_{q_1,\lambda})}$,

\begin{equation*}
\var(\tilde F_\lambda-\tilde F_{q_{1},\lambda})\asymp \sum_{j>q_{1}}\var\left(\frac{F_{j,\lambda}}{\lambda^{k-q_{1}/2}}\right) \leq C_{1}{{\lambda}}^{-1},
\end{equation*}

for some finite constant $C_1$ independent of $\lambda$.

\begin{enumerate}
\item Assume $h_{1}\neq 0$ and write $\tilde f_{1}:=\|f_{1}\|_{L^2(Z;\mu_{\lambda})}^{-1}f_{1}$. We can use Theorem \ref{t:main2} in the case $k=1$ and $q_{1}=1$ (together with (\ref{eq:Lp-scaling})), in order to deduce the bound
\begin{equation*}
d_{W}(\tilde F_{1,\lambda},N) \leq c\|\tilde f_{1}\|_{3}^3 =c\frac{\lambda^{1+3(k-1)}\|h_{1}\|_{L^3}^3}{(\lambda^{1+2(k-1)}\|h_{1}\|_{L^2}^2)^{3/2}}=C_{2} \lambda^{3k-2-(3/2)(2k-1)}=C_{2}\lambda^{-1/2},
\end{equation*}
for some constant $C_2$ independent of $\lambda$. The required estimate follows from the standard inequality $d_W(\tilde{F}_\lambda, N) \leq d_{W}(\tilde F_{1,\lambda},N) + \var(\tilde F_\lambda-\tilde F_{1,\lambda})^{1/2}$, $N\sim \mathscr{N}(0,1)$,
as well as from the estimates contained in the first part of the proof.

\item For every integer $M\geq 1$, we write $[M] = \{1,...,M\}$. Assume that $q_1>1$, implying that, for every $i<q_1$ 
\[
\int_{Z^{k-i}} h({\bf x}_i, \bfx_{k-i}) \mu^{k-i}(d\bfx_{k-i}) = 0,
\]
almost everywhere $d\mu^{i}$. Let $\{x_1,...,x_{\eta_\lambda(Z)}\}$ be any enumeration of the support of $\eta_\lambda$ inside $Z$, and observe that
\[
F_\lambda = k! \sum_{\{i_1,...,i_k\}\subset [\eta_\lambda(Z)]} h(x_{i_1},...,x_{i_k}).
\]
The standard theory of Hoeffding decompositions (see e.g. Vitale \cite{Vitale}) implies that $F_\lambda$ admits a unique decomposition of the form
\[
F_\lambda = k! \sum_{i=q_1}^k \binom{\eta_\lambda(Z) -i }{k-i} \sum_{\{j_1,...,j_i\} \subset  [\eta_\lambda(Z)] }H_i(x_{j_1},...,x_{j_i}),
\]
where $H_{q_1} = h_{q_1} \binom{k}{q_1}^{-1} \mu(Z)^{q_1-k}$, and each $H_i$ is a degenerate kernel in $i$ variables -- in the sense that $\int_{Z} H_i({\bf x}_{i-1}, x) \mu(dx) = 0$, almost everywhere $d\mu^{i-1}$. The strong law of large numbers (which is a consequence of the infinite divisibility of the Poisson distribution) implies that, for every $i=q_1,...,k$, 
\[
\frac{k! \binom{\eta_\lambda(Z) -i }{k-i}}{\lambda^{k-i}} \rightarrow \mu(Z)^{k-i}\frac{k!}{(k-i)!},
\]
in probability, as $\lambda\to\infty$. 
Using \cite[Theorem 2]{DyMa}, we also deduce that, for every $i=q_1,...,k$, the class
\[
\frac{1}{\lambda^{i/2}}\sum_{\{j_1,...,j_i\} \subset  [\eta_\lambda(Z)] }H_i(x_{j_1},...,x_{j_i}), \quad \lambda>0,
\]
converges to 
\[
\mu(Z)^{i/2} \frac{1}{i!} I^{\hat G}_i(H_i),
\]
and the conclusion follows immediately.
\end{enumerate}
\end{proof}

\section{Applications}\label{s:applications}

The methods developed in the present article and in \cite{LacPec} are well tailored for dealing with spatial processes enjoying some translation-invariance properties. We present below some applications. We first consider the total mass of weighted edges of the random graph based on a fundamental object of stochastic geometry: the boolean model. Then, we show how our results can be applied to the study of a random telecommunication network, where devices are placed at random locations and have random radii of interaction.

On the boolean model we consider a classical $2$d-order \Us, but the coverage of the random telecommunication network is quantified in terms of $k$-th order \Us, for $k \geq 1$.
\medskip

\noindent Throughout this section, we use the notation

$$
X_\lambda = [-\lambda^{1/d}, \lambda^{1/d}]^d,
$$
where $d\geq 1$ is some integer.

\subsection{Random graph defined on a boolean model}

The standard boolean model consists of iid random compact sets (grains) disposed at random locations (germs) in the space. The union of such sets stands out as one of the most used and most studied   example of a stationary random closed set of $\mathbb{R}^d$. The reference books \cite{Mat75, SchWei, SKM} provide a good presentation of the model and of its applications in epidemiology, forestry, or image synthesis. The specificity of the graph based on the boolean model, where two germs are connected if their respective grains intersect, is that  the geometric behaviour of individuals is random, meaning the geometric  rule of interaction is not determined in advance.  Limit theorems  are discussed for instance in \cite{HeiMol, Mol}, in view of estimating parameters of the models like for instance the intensity of the the process or specific geometric quantities related to the typical grain. Molchanov and Heinrich \cite[Section 9]{HeiMol} mention that central limit theorems for second and higher order quantities are possible to obtain via a laborious application  of their method. We develop here a simple procedure to deal with a random variable based on pairwise interactions of the model, that can easily be extended to higher orders.

We now proceed with a formal description. In what follows, the mark space is the class $\mathcal{K}$ of compact subsets of $\mathbb{R}^d$, endowed with the Fell topology and the Borel $\sigma$-algebra (see \cite{SchWei} for definition and properties). We shall denote by $\eta$ a marked Poisson measure with values in $\mathbb{R}^d\times \mathcal{K}$ and control measure $\mu=\leb\otimes \nu$, where $\nu$ is a probability distribution on $\mathcal{K}$.  Of actual interest for modeling purposes is the random closed set defined by 
\begin{equation*}
F_{\eta}=\bigcup_{(x,C)\in \eta}(C+x),
\end{equation*}
called \emph{boolean model}  with grain distribution $\nu$. We call  \emph{typical grain} a random compact set with distribution $\nu$. Denote by $\eta^\circ=\{x:(x,C)\in \eta\text{ for some $C\in\Cpt$}\}$ the ground-process, and for every $x\in \eta^\circ$,   $C=C_{x}$ the unique grain such that $(x,C)\in \eta$.
 For $\lambda>0$, we call $\eta_{\lambda}$ (resp. $\eta_{\lambda}^\circ$) the restriction of $\eta$ to $Z_{\lambda}=X_{\lambda}\times \Cpt$ (resp. $X_{\lambda}$), and $F_{\eta_{\lambda}}$ the corresponding random closed set, obtained from $F_{\eta}$ by removing all grains whose center lie outside $X_{\lambda}$. Remark that $F_{\eta_{\lambda}}\neq F_{\eta}\cap Z_{\lambda}$ as $x\notin X_{\lambda}$ does not imply that $(C_{x}+x)\cap X_{\lambda}=\emptyset$. 

\smallskip

In what follows, we consider the asymptotic normality here of variables of the type
\begin{equation}
\label{eq:boolean-model}
G_{\lambda}=\sum_{x\neq y \in \eta^\circ_{\lambda}}\varphi(x-y)1_{\{(x+C_{x})\cap (y+C_{y})\neq\emptyset\}}
\end{equation} 
where $\varphi$ is some even real function on $\mathbb{R}^d$ such that $\varphi^2$ is integrable on every compact set (note that these assumptions imply that $G_{\lambda}$ has a finite second moment). 

%
%
%

 If for instance $\varphi(x,y)=\|x-y\|$, $G_{\lambda}$ is the total length of the edges of the graph where two points of $X_{\lambda}$ in the ground process $\eta_{\lambda}^\circ$ are connected if their respective grains touch. Notice that the techniques used here could be used as well for a kernel $\varphi(x-y; C,C')$ where the weight for each pair of points $(x,y)$ depends also on the grains $C_{x}$ and $C_{y}$ actually attached to $x$ and $y$.
  
The variable $G_{\lambda}$ is  a \Us $\,$ on $Z_{\lambda}=X_{\lambda}\times \Cpt$ of order $k=2$ with kernel $h(x,y;C,C')=\varphi(x-y)1_{\{(x+C)\cap (y+C')\neq \emptyset\}}$ on $Z_{\lambda}^2$. The kernel $h$  indeed belongs to $L^{1,2}(Z_{\lambda}^2)$ because
 \begin{equation*}
\int_{Z_{\lambda}^2}|h_{\lambda}(x,y;C,C')|^p dxdy d\nu(C )d\nu(C')\leq   \int_{X_{\lambda}^2}|\varphi(x-y)|^pdxdy<\infty
\end{equation*}
for $p=1,2$. Furthermore $h$ is stationary with respect to its spatial variables, with factorization given by
\begin{equation*}
\fact h(x; C_{1}, C_{2})=\varphi(x)1_{\{C_{1}\cap (x+C_{2})\neq\emptyset\}}.
\end{equation*}

Let us introduce the probability that two independent germs centered respectively at $0$ and $x$ have a non-empty intersection
\begin{equation*}
\chi_{\nu}(x)=P( C_{1}\cap (x+C_{2})\neq\emptyset), x\in \mathbb{R}^d
\end{equation*}
where $C_{1}$ and $C_{2}$ are two iid variables in $\Cpt$ with distribution $\nu$. 
It can be shown in the framework of stabilization theory (see for instance \cite{BarYuk05}, Theorem 2.1) that under reasonable decay assumptions on $\chi_{\nu}$ we have
\begin{equation*}
\lim_{\lambda\to \infty}E \lambda^{-1}G_{\lambda}=  \int_{\mathbb{R}^d}\varphi(x) \chi_{\nu}(x)dx.
\end{equation*}

\begin{thm}
\label{thm:general-booleanmodel}
Assume that $\varphi$ and $\nu$ satisfy, for some $\varepsilon>0$, and for some bounded strictly positive probability density,
\begin{equation}
\label{eq:general-condition-boolean}
\int_{}\varphi(x)^p \chi_{\nu}(x)\kappa(x)^{1-p}dx<\infty
\end{equation}
for $p=2,4$. Then $\var(G_{\lambda})\sim c\lambda$ with $c>0$ and 
\begin{equation*}
d_{W}(\tilde G_{\lambda},N) \leq C \lambda^{-1/2}
\end{equation*}
for some constant $C,c>0$ not depending on $\lambda$.
\end{thm}

\begin{proof}
Let us apply directly Theorem \ref{th:CLT-Ustat} with the the rescaling $\alpha_{\lambda}=\lambda^{-1/d}$ on $X_{1}$ (recall the formulation of Problem \ref{prob:Ustat}, as well as Example \ref{ex:x}-(ii)). A condition for asymptotic normality with a bound on the Wasserstein distance decreasing at the speed of $\lambda^{-1/2}$ is that $h$ has rapidly decreasing projections. Using Lemma \ref{lm:integrability-condition} , the finiteness of (\ref{e:akappa}) boils down to
\begin{equation*}
\int_{}|\fact h(x;C,C')|^p \kappa(x)^{1-p}dx<\infty,\quad p=2,4,
\end{equation*}
which directly is given by (\ref{eq:general-condition-boolean}).
\end{proof}

\begin{rem}{\rm 

Since $\chi_{ \nu}\leq 1$, a sufficient condition in order to have (\ref{eq:general-condition-boolean}) is that the function $\varphi$ converges to zero at a subexponential rate, as $\|x\|_{\R^d}\to \infty$. The next statement contains an example about how one can select a measure $\nu$ allowing to deal with a possibly increasing $\varphi$. 

}
\end{rem}

\begin{prop}
Assume that $\varphi(x,y)=\|x-y\|^\beta$ for some $\beta >-d/2$ and that the typical   grain is comprised in a ball with random radius $R$ with distribution $\nu(dr)=C_{\alpha}'1_{\{r\geq 1\}}r^{-\alpha} dr$, $\alpha>1$. Then, if 
\begin{equation}
\label{eq:alpha-boolean-condition}
\alpha>2(\beta+d)+1,
\end{equation}
the conclusion of Th. \ref{thm:general-booleanmodel} holds.
\end{prop}

\begin{proof}
Let $\varepsilon>0$ be such that  $\beta p+(1-\alpha)+(d+\varepsilon)p-\varepsilon<0$ for $p=2,4$ (the inequality for $p=4$ is sufficient),
and let 
\begin{equation*}
\kappa(x)=\frac{C_{d,\varepsilon}}{1+\|x\|^{d+\varepsilon}}, x\in \mathbb{R}^d,
\end{equation*}
with the appropriate normalisation constant $C_{d,\varepsilon}$.
 We have 
\begin{equation*}
\chi_{\nu}(x) \leq P(R+R'\leq \|x\|)
\end{equation*} where $R'$ is an idependent copy of $R$, which gives $\chi_{\nu}(x) \leq C\|x\|^{1-\alpha}$ for some constant $C>0$. Then, (\ref{eq:general-condition-boolean}) is implied by  
\begin{equation*}
\int_{1}^{+\infty}r^{\beta p} r^{1-\alpha} r^{ {(d +\varepsilon)(p-1)}} r^{d-1}dr<\infty
\end{equation*}
for $p=2,4$.\end{proof}

The functional $G_{\lambda}$ is stabilizing, meaning it can be written as a sum of contributions over the points of $\eta_{\lambda}^\circ$, where the contribution of each point only depends on the intersection of $\eta_{\lambda}^\circ$ with a random ball centered in this point (for a proper introduction and summary of results, see \cite[Chapter 4]{KenMol10} and the references therein). Our results can be compared e.g. with the very general Theorems 2.1 and 2.2 in \cite{Pen07}, that imply in this specific framework the asymptotic normality of $\lambda^{-1/2}(G_{\lambda}-EG_{\lambda})$ if 
$\alpha-1>150 d$ (and an additional assumption on $\varphi$), giving a more restrictive  condition than (\ref{eq:alpha-boolean-condition}).
Note that, to our knowledge, there is no rate of convergence for CLTs involving such polynomially stabilizing functionals in the literature.

 
\subsection{Coverage of a telecommunications network with random user range}

The content of this section is inspired by reference \cite{DFR}, where Decreusefond {\it et al.} use concepts from algebraic topology in order to study the asymptotic interactions between devices in a telecommunications network. 

\medskip

\noindent In what follows, the devices are spread according to a Poisson measure on $X_{\lambda}=[-\lambda^{1/d}, \lambda^{1/d}]$, with intensity given by the Lebesgue measure. Two devices communicate within the network if their distance is smaller than a constant, say $r$. To account for the good coverage of the network by the devices, they consider, for each $k\geq 2$,  the number $N_{k,\lambda}$ of {\it $(k-1)$-simplices}, i.e. of the number of $k$-tuples $(x_{1},\dots,x_{k})$ inside $X_\lambda^k$ whose pairwise distances are smaller than $r$, meaning that the group of $k$ devices is indeed connected. Among other findings, the authors of \cite{DFR} give closed expressions for the moments of $N_{k,\lambda}$, and derive a general CLT. One should note that in \cite{DFR} the metric of the torus is used in order to avoid edge effects. The problem can also be seen as a generalization of subgraph counting (see Section \ref{s:sub}) with random radii of interactions, for the particular  graph of the $k$-simplex. 

\medskip

\noindent The results contained in Section \ref{sec:invariant} of the present paper  are precisely tailored to deal with edge effects for stationary spatial processes, thus we shall perform an analogous study by replacing  the torus distance with the euclidean distance. We also introduce the assumption that members of the network have a random range, i.e. we assume that to each member $x_{i}$ of the network is assigned a random radius $R_{i}$, that the radii are independent and identically distributed, and that a simplex of $k$ members is ``connected''  within the network if the ball around each  member $x_{i}$ contains every other member of the $k$-tuple. 

\medskip

Formally, we shall consider a marked Poisson process $\eta$ over $\R^d\times \R_+$, with intensity given by $\ell(dx)\otimes \nu$, where $\nu$ is some probability measure on $\R_+$. Every realization of $\eta$ is therefore a collection of points of the type $(x_i,R_i)$, where every $R_i$ indicates a random radius with law $\nu$. For every $\lambda$, the number $N_{k,\lambda}$ we are interested in counts the number of those $\{x_1,...,x_k\}\subset X_\lambda$ that are in the support of $\eta$ and such that $\|x_{i}-x_{j}\|_{\R^d} \leq R_i$, for every $i,j=1,...,k$. The following statement provides sufficient conditions to have a standard limit behavior (variance in $\lambda$ and convergence to the normal law at a speed at most of $\lambda^{-1/2}$ in the Wasserstein distance).

\begin{thm}
Denote by $F$ the tail function of the law of $R_{1}$, that is: $F(r) = P[R_1>r] = \nu((r, \infty)]$.  
If for some $\varepsilon>0$
\begin{equation*}
\int_{\mathbb{R}_{+}}F(r) r^{4d-1+\varepsilon}dr<\infty,
\end{equation*}
we have, for every $k \geq 2$ and as $\lambda\to\infty$,  \begin{align*}
\var(N_{k,\lambda})&\sim c_{k} \lambda\\
d_{W}(\tilde N_{k,\lambda},N)&\leq C_{k}\lambda^{-1/2}
\end{align*}
for some $C_{k},c_{k}>0$ independent of $\lambda$. Here, $\tilde N_{k,\lambda}$ indicates a centered and renormalized version of $N_{k,\lambda}$, whereas $N$ is a centered Gaussian random variable with unit variance.
\end{thm}

\begin{proof} The random variable
$N_{k,\lambda}$ is the \Us $\,$ of order $k$ on $(X_{\lambda}\times \mathbb{R}_{+})^k$ with symmetric kernel
\begin{equation*}
h((x_{i},r_{i})_{i=1,\dots,k})=\1_{\{\|x_{i}-x_{j}\| \leq r_{i};\quad 1 \leq i,j \leq k\}}.
\end{equation*}
Such a kernel is clearly stationary with respect to spatial translations. Let   $\varepsilon'=\frac{\varepsilon}{(k-1)} $ and put  $\kappa(x)\asymp (1+\|x\|)^{-(d+\varepsilon')} $.
According to Theorem \ref{th:CLT-Ustat} and Lemma  \ref{lm:integrability-condition}, a sufficient condition for the result  is 
\begin{equation*}
I=\int_{\mrk{\Zi^{k-1}}}\fact h^p(\mrk \bfx_{k-1})\prod_{i=2}^k\|x_{i}\|^{(d+\varepsilon')(p-1)} d\mrk\mu(\mrk \bfx_{k-1})<\infty
\end{equation*}
for $p=2,4$. We have, setting $x_1 =0$,
\begin{align*}
I&=\int_{(\mathbb{R}^d)^{k-1}}\prod_{i=2}^k\|x_{i}\|^{(d+\varepsilon')(p-1)}\int_{\mathbb{R}_{+}^k}\prod_{i=1}^k\1_{\{\rho_{i} \geq \max_{l=1\dots k}\|x_{i}-x_{l}\|\}} d\nu(\rho_{1})\dots d\nu(\rho_{k})d\bfx_{k-1}\\
&\leq\int_{(\mathbb{R}^d)^{k-1}}\prod_{i=2}^k\left(\|x_{i}\|^{(d+\varepsilon')(p-1)}\int_{\mathbb{R}_{+}}\1_{\{\rho_{i} \geq \|x_{i}\|\}} d\nu(\rho_{i})\right)d\bfx_{k-1}\\
&\leq  \left(\int_{\mathbb{R}_{+}}r^{(p-1)(d+\varepsilon')+d-1}  F( r)dr\right)^{k-1},\\
\end{align*}
after doing a spherical change of variables for each $x_{i}$, whose Jacobian is smaller than $\|x_{i}\|^{d-1}$, 
which concludes the proof because  $(p-1)(d+\varepsilon')+d-1 =pd-1+\varepsilon \leq 4d-1+\varepsilon$.\\
\end{proof}

\bibliographystyle{plain}

\end{document}